\documentclass[a4paper,11pt,twoside]{article}

\usepackage{geometry}
\usepackage[utf8]{inputenc}
\usepackage{lmodern}
\usepackage[T1]{fontenc}
\usepackage{dsfont}
\usepackage{tikz}
\usepackage{amsmath,amssymb,amsthm,empheq,cases}
\usepackage{mathrsfs}
\usepackage{graphicx}
\usepackage{url}
\usepackage{hyperref}
\hypersetup{colorlinks,
            citecolor=red, 
            filecolor=black,
            linkcolor=blue,
            urlcolor=black}

\def \dis {\displaystyle}

\def \ecart {\noalign{\medskip}}

\newtheorem{theorem}{Theorem}
\newtheorem{corollary}[theorem]{Corollary}
\newtheorem{definition}[theorem]{Definition}
\newtheorem{lemma}[theorem]{Lemma}
\newtheorem{proposition}[theorem]{Proposition}
\newtheorem{remark}[theorem]{Remark}

\def \refL #1{Lemma~\ref{#1}}
\def \refC #1{Corollary~\ref{#1}}
\def \refP #1{Proposition~\ref{#1}}
\def \refR #1{Remark~\ref{#1}}

\title{Semigroups generated in $L^{p}$-spaces by some dispersal process
including semi-permeability conditions at the interface}

\author{La\"{i}d Djilali, Rabah Labbas, Ahmed Medeghri, Abdallah Menad
\& Alexandre Thorel \\
{\scriptsize R. L and A.T: Université Le Havre Normandie, Normandie Univ, LMAH UR 3821, 76600 Le Havre, France.}\\
{\scriptsize rabah.labbas@univ-lehavre.fr, alexandre.thorel@univ-lehavre.fr}\\
{\scriptsize L. D and A. M: Laboratoire de Math\'{e}matiques
Pures et Appliqu\'{e}es,}\\
{\scriptsize Universit\'{e} Abdelhamid Ibn Badis, 27000 Mostaganem, Alg\'{e}rie,}\\
{\scriptsize A. M: National Higher School of Mathematics, Scientific and Technology Hub of Sidi Abdellah,}\\
{\scriptsize P.O. Box 75, Algiers 16093, Algeria.}\\
{\scriptsize laid.djilali@univ-mosta.dz, ahmed.medeghri@nhsm.edu.dz, abellah.menad@univ-mosta.dz}}
\date{\empty }

\begin{document}

\maketitle

\begin{abstract}
We study an elliptic differential equation set in two habitats under semi-permeability conditions at the interface. This equation describes some dispersal process in population dynamics. Using the well-known Dore-Venni theorem, some useful results in \cite{DoreLabbas} and \cite{weis}, we show that the associated space operator generates an analytic semigroup in $L^{p}$-spaces.

\medskip

\end{abstract}

\section{Introduction}

In this work, we are interested in the study of a reaction-diffusion model for population dynamics with a dispersal process in two habitats. This model takes into account the reaction effect of individuals on the interface of these habitats. 

Our aim is to show that this dispersal process generates an analytic semigroup in the framework of $L^{p}$-spaces. The present work completes naturally the results obtained in \cite{FainiLabbasMedeghriMenad}.

Many authors have worked on different models of reaction-diffusion problems
related to biology or various environmental phenomena. We were inspired, in particular, by the work in \cite{Cantrell}, where the model considered incorporates the response of individuals on the interface between the habitats. This study was detailed in one space dimension and focused mainly on the spectral aspect of the dispersal process, taking into account the dimension of the habitats.

Here, the concrete example which illustrates our theory, is modelled by partial differential equations of parabolic type set in two juxtaposed habitats $\Omega _{-}$ and $\Omega _{+}$: 
\begin{equation*}
\Omega =\Omega _{-}\cup \Gamma _{0}\cup \Omega _{+},
\end{equation*}
where
\begin{equation}
\left\{ 
\begin{array}{l}
\Omega _{-}=(-\ell ,0) \times (0,1), \medskip \\ 
\Gamma _{0}=\left\{ 0\right\} \times (0,1), \medskip \\ 
\Omega _{+}= (0,L) \times (0,1), \medskip \\ 
\Gamma _{\pm }=\partial \Omega _{\pm }\backslash \Gamma _{0},
\end{array}
\right.  \label{Landscape}
\end{equation}
and $\ell $, $L>0$. The reaction-diffusion equation is
\begin{equation}
\dfrac{\partial u}{\partial t}(t,x,y)=\left\{ 
\begin{array}{ll}
d_{-}\Delta u_{-}(t,x,y)+F_{-}(u_{-}(t,x,y)) & \text{in } (0,T)
\times \Omega _{-},\medskip \\ 
d_{+}\Delta u_{+}(t,x,y)+F_{+}(u_{+}(t,x,y)) & \text{in } (0,T)
\times \Omega _{+},
\end{array}\right.  \label{Equation}
\end{equation}
under the initial data
\begin{equation}
u(0,.)=\left\{ 
\begin{array}{ll}
\varphi _{-} & \text{in }\Omega _{-}, \\ 
\varphi _{+} & \text{in }\Omega _{+},
\end{array}
\right.  \label{InitialData}
\end{equation}
the boundary conditions 
\begin{equation}
\left\{ 
\begin{array}{ll}
u_{-}=0 & \text{on }\Gamma _{-}, \\ 
u_{+}=0 & \text{on }\Gamma _{+},
\end{array}\right.  \label{BoundaryConditons}
\end{equation}
and the interface conditions
\begin{equation}
\left\{ 
\begin{array}{ll}
d_{-}\dfrac{\partial u_{-}}{\partial x}=q\left( u_{+}-u_{-}\right) & \text{on }\Gamma _{0}, \medskip \\ 
d_{+}\dfrac{\partial u_{+}}{\partial x}=q\left( u_{+}-u_{-}\right) & \text{on }\Gamma _{0};
\end{array}\right.  \label{TransmissionCdts}
\end{equation}
where $d_{\pm } > 0$ is the diffusion coefficient, $q>0$ is given and $u_{\pm }$ represents a population density in $\Omega_\pm$. 

These two last conditions in \eqref{TransmissionCdts} mean that the flux at the
interface depends on the density of the population. They are of the Robin
type and express the semi-permeability of the interface \ $\Gamma _{0}$.

In \cite{FainiLabbasMedeghriMenad}, the authors have considered different
interface conditions in the framework of the space of continuous functions
which are the following: the non-continuity of the flux and the continuity
of the dispersal at the interface: 
\begin{equation*}
\left\{\begin{array}{ll}
(1-\alpha )d_{-}\dfrac{\partial u_{-}}{\partial x}=\alpha d_{+}\dfrac{\partial u_{+}}{\partial x} & \text{on }\Gamma _{0}, \medskip \\

d_{-}\Delta u_{-}+F_{-}(u_{-})=d_{+}\Delta u_{+}+F_{+}(u_{+}) & \text{on }\Gamma _{0},
\end{array}\right.
\end{equation*}
where $\alpha \in (0,1)$ with $\alpha \neq 1/2$.

Here, we consider only the linearized part of logistic functions, that is
\begin{equation*}
\left\{ 
\begin{array}{ll}
F_{-}(u_{-})=r_{-}u_{-} & \text{on } (-\ell ,0) \times (0,1), \\ 
F_{+}(u_{+})=r_{+}u_{+} & \text{on } (0,L) \times (0,1),
\end{array}\right.
\end{equation*}
where $r_\pm > 0$.

The study of the above reaction-diffusion problem depends essentially on the nature
of operator $\mathcal{P}$ defined, in the stationary case, by
\begin{equation*}
\left\{ 
\begin{array}{l}
D\left( \mathcal{P}\right) = \left\{ 
\begin{array}{l}
u\in L^{p}(\Omega ):\text{ }u_{-}\in W^{2,p}\left( \Omega _{-}\right) ,\text{
}u_{+}\in W^{2,p}\left( \Omega _{+}\right) ,\text{ }u_{\pm }=0\text{ on }
\Gamma _{\pm } \\ 
\text{and }u_{\pm }\text{ satisfies (\ref{TransmissionCdts})}
\end{array}
\right\}, \medskip \\ 
\mathcal{P}u=\left\{ 
\begin{array}{ll}
d_{-}\Delta u_{-}-r_{-}u_{-} & \text{in }\Omega _{-}, \\ 
d_{+}\Delta u_{+}-r_{+}u_{+} & \text{in }\Omega _{+},
\end{array}\right.
\end{array}\right.
\end{equation*}
where $p \in (1,+\infty)$. Note that the transmission conditions \eqref{TransmissionCdts} are well defined since, for example, if $u_{-}\in W^{2,p}(\Omega _{-})$, then 
\begin{equation*}
\frac{\partial u_{-}}{\partial x}|_{\left\{ 0\right\} \times (0,1)} \in W^{1-1/p,p}(0,1),
\end{equation*}
see \cite{Grisvard}, Corollary 1, p. 682.

Now, let us write the above parabolic equation in an abstract formulation.
First, let us introduce, in the Banach space $E_{0}=L^{p}\left( 0,1\right)$,
operator $A_{0}$ defined by
\begin{equation}
\left\{ 
\begin{array}{l}
D(A_{0})=\left\{ \varphi \in W^{2,p}(0,1):\varphi (0)=\varphi (1)=0\right\},
\\ 
\left( A_{0}\varphi \right) (y)=\varphi ^{\prime \prime }(y).
\end{array}
\right.  \label{DefinitionAzero}
\end{equation}
It is known that this operator is closed linear with a dense domain and
verifies the two following properties :
\begin{equation}
\exists \,C>0:\forall z\in S_{\pi -\varepsilon}\cup \left\{ 0\right\}
, \quad \left\Vert (zI-A_{0})^{-1}\right\Vert _{\mathcal{L}(E_{0})}\leqslant \dfrac{C}{
1+\left\vert z\right\vert },  \label{EllipticityOfA0}
\end{equation}
where
\begin{equation*}
S_{\pi -\varepsilon}=\left\{ z\in \mathbb{C}\backslash \left\{ 0\right\} :\left\vert \arg z\right\vert <\pi - \varepsilon\right\} ,
\end{equation*}
for any small $\varepsilon > 0$ such that $\pi -\varepsilon > 0$ and
there exists a ball $B(0,\delta )$, such that $\overline{B(0,\delta )} \subset \rho (A_{0})$ and the above estimate is still true in $S_{\pi -\varepsilon}\cup \overline{B(0,\delta )}$; and
\begin{equation}\label{BipA0}
\forall s\in \mathbb{R},~ \left( -A_{0}\right) ^{is}\in \mathcal{L}\left(E_{0}\right) ,~ \forall \,\varepsilon > 0 : \underset{s\in \mathbb{R}}{\sup }\left\Vert e^{-\varepsilon \left\vert s\right\vert }(-A_{0})^{is}\right\Vert _{\mathcal{L}\left( E_{0}\right) }<+\infty,
\end{equation}
see for instance the method used in \cite{labbas-moussaoui}, Proposition 3.1, p. 191.
\begin{remark}
The above problem can be considered in dimension $n$ by setting
$$\Omega = (-\ell, L)\times \varpi,$$
where $\varpi \subset \mathbb{R}^{n-1}$, $n>1$, is a bounded regular open set, $E_0 = L^p(\varpi)$ and 
\begin{equation*}
\left\{ 
\begin{array}{l}
D(A_{0})=\left\{ \varphi \in W^{2,p}(\varpi): \varphi|_{\partial \varpi} = 0\right\}, \\ 
A_{0}\varphi =\Delta_\varpi \varphi;
\end{array}\right.  
\end{equation*}
here $\Delta_\varpi$ denotes the Laplace operator related to variables on $\varpi$.
\end{remark}
We will also use the following usual operational notation of vector-valued
functions:
\begin{equation*}
u_{\pm }(t,x)(y):=u_{\pm }(t,x,y),\quad t\in (0,T),~ (x,y)\in \Omega _{\pm }.
\end{equation*}
So, we have to analyze the abstract Cauchy problem
\begin{equation}
\left\{ 
\begin{array}{l}
u^{\prime }(t)=\mathcal{S}_{0}u(t), \\ 
u(0)=u_{0},
\end{array}
\right.  \label{CauchyProblem}
\end{equation}
set in the Banach space $L^{p}(-\ell ,L;E_{0})$, where
$$\left\{\begin{array}{cll}
D\left( \mathcal{S}_{0}\right) &=&\left\{ 
\begin{array}{l}
w\in L^{p}(-\ell ,L;E_{0}):~w_{-}\in W^{2,p}(-\ell ,0;E_{0}),~w_{+}\in
W^{2,p}(0,L;E_{0}), \\ 
w_{-}\in L^p(-\ell ,0;D(A_{0})),~w_{+}\in L^p(0,L;D(A_{0})),  \\
w_{-}(-\ell )=0,~w_{+}(L)=0,~d_{-}w'_{-}(0)=q\left(w_{+}(0)-w_{-}(0)\right) \\ 
\text{and }d_{+}w_{+}^{\prime }(0)=q\left( w_{+}(0)-w_{-}(0)\right)
\end{array}\right\}, \medskip \\
\left( \mathcal{S}_{0}w\right) (x) &=&\left\{ 
\begin{array}{ll}
d_{-}w_{-}^{\prime \prime }(x)+d_{-}A_{0}w_{-}(x)-r_{-}w_{-}(x) & \text{in } (-\ell ,0), \\ 
d_{+}w_{+}^{\prime \prime }(x)+d_{+}A_{0}w_{+}(x)-r_{+}w_{+}(x) & \text{in } (0,L).
\end{array}\right.
\end{array}\right.$$
We then consider a more general operator $\mathcal{S}$ instead of $\mathcal{S}_{0}$ where $A_{0}$ is replaced by a closed linear operator $A$ in a Banach space $E$ satisfying some assumptions specified in Section \ref{Sect Assumptions and main results}.

Our method is essentially based on the use of abstract differential equations of elliptic type. The merit of this method lies above all in the fact of having the explicit formula of the resolvent operator of $\mathcal{S}$ (and therefore of $\mathcal{S}_0$), see Section \ref{Sect 6}.

This paper is organized as follows.

In Section \ref{Sect 2}, we only show that problem $-\mathcal{P}u=g$ can be studied in the variational framework. In Section \ref{Sect 3}, we recall some useful notions on sectorial operators. Section \ref{Sect Assumptions and main results} contains our assumptions and the main results. Section \ref{Sect 5} is devoted to the establishment of some useful properties on complex numbers. Section \ref{Sect 6} is composed of two subsections. In the first subsection, we explain the spectral equation of $\mathcal{S}$ which leads to an abstract system to be solved. Thus, we invert the determinant operator of the above system by using among others the $H^\infty$-calculus. This leads us to obtain the explicit solution of the spectral equation. Many of properties and techniques used in \cite{DoreLabbas} were useful to us in this work. Then, we study the optimal regularity of this solution. In the second subsection, we give some sharp estimates which lead us to analyze the behaviour of the resolvent operator of $\mathcal{S}$. We then obtain our main results which state among others that $\mathcal{S}$ generates an analytic semigroup in $L^p(-\ell,L;E)$, for $p \in (1,+\infty)$.

\section{Variational formulation of $-\mathcal{P}u=g$}\label{Sect 2}

Let $a,b \in \mathbb{R}$, with $a<b$. For any
\begin{equation*}
\begin{array}{cccl}
\varphi : & (a,b) & \longrightarrow & \mathbb{R} \\ 
& x & \longmapsto & \varphi \left( x\right) ,
\end{array}
\end{equation*}
we set 
\begin{equation*}
\begin{array}{cccl}
\widetilde{\varphi }: & \mathbb{R} & \longrightarrow & \mathbb{R} \\ 
& x & \longmapsto & \widetilde{\varphi }\left( x\right) =\left\{ 
\begin{array}{ll}
\varphi \left( x\right) & \text{for }x\in (a,b), \\ 
0 & \text{for }x\in \mathbb{R}\backslash (a,b),
\end{array}
\right.
\end{array}
\end{equation*}
and we define a subspace $\widetilde{H^{1/2}}\left( a,b\right) $ of $
H^{1/2}(a,b)$ by
\begin{equation*}
\widetilde{H^{1/2}}\left( a,b\right) =\left\{ \varphi \in H^{1/2}\left(
a,b\right) :\widetilde{\varphi }\in H^{1/2}\left( \mathbb{R}\right) \right\}.
\end{equation*}
In \cite{Lions Magenes}, the authors denoted this space by $
H_{00}^{1/2}\left( a,b\right) $ which also coincides with the following
particular interpolation space 
\begin{equation*}
\left( H_{0}^{1}(a,b),L^{2}(a,b)\right) _{1/2,2}.
\end{equation*}
The interpolation spaces are described, for instance, in \cite{Grisvard}.
 
Set 
\begin{equation*}
H_{\Gamma _{\pm }}^{1}\left( \Omega _{\pm }\right) =\left\{ u_{\pm }\in
H^{1}\left( \Omega _{\pm }\right) :u_{\pm \mid \Gamma _{\pm }}=0\right\} ;
\end{equation*}
for $v_{\pm }$ in $H_{\Gamma _{\pm }}^{1}\left( \Omega _{\pm }\right) $, it
is clear that $v_{\pm \mid \Gamma _{0}}$ is in $\widetilde{H^{1/2}}\left(
\Gamma _{0}\right) .$

Problem $-\mathcal{P}u=g$ writes in the form 
\begin{equation*}
\left\{ 
\begin{array}{rcll}
-\text{div}\left( d_{+}\nabla u_{+}\right) +r_{+}u_{+} & = & g_{+} & \text{in }\Omega _{+},\medskip \\ 
-\text{div}\left( d_{-}\nabla u_{-}\right) +r_{-}u_{-} & = & g_{-} & \text{in }\Omega _{-},\medskip \\ 
\displaystyle d_{+}\frac{\partial u_{+}}{\partial \nu } & = & q\left( u_{+}-u_{-}\right) & \text{on }\Gamma _{0},\medskip \\ 
\displaystyle d_{-}\frac{\partial u_{-}}{\partial \nu } & = & q\left( u_{+}-u_{-}\right) & \text{on }\Gamma _{0},\medskip \\ 
u_{\pm } & = & 0 & \text{on }\Gamma _{\pm },
\end{array}
\right.
\end{equation*}
where $\nu $ is the normal unit vector oriented towards the interior of $
\Omega _{+}$. The variational formulation is set in the hilbertian space
\begin{equation*}
\mathbb{V}=H_{\Gamma _{+}}^{1}\left( \Omega _{+}\right) \times H_{\Gamma
_{-}}^{1}\left( \Omega _{-}\right) ,
\end{equation*}
with
\begin{eqnarray*}
a\left( \left( u_{+},u_{-}\right) ,\left( w_{+},w_{-}\right) \right) &=&\displaystyle \int\limits_{\Omega _{+}}\left( d_{+}\nabla u_{+} \cdot \nabla w_{+} + r_{+}u_{+}w_{+}\right) dxdy \\
&&+\displaystyle \int\limits_{\Omega _{-}}\left(
d_{-}\nabla u_{-} \cdot \nabla w_{-} + r_{-}u_{-}w_{-}\right) dxdy,
\end{eqnarray*}
and
\begin{equation*}
\left\{ 
\begin{array}{l}
b\left( \left( u_{+},u_{-}\right) ,\left( w_{+},w_{-}\right) \right)
=\displaystyle \int\limits_{\Gamma _{0}}q\left( u_{+}-u_{-}\right) \left(
w_{+}-w_{-}\right) dy, \\ 
l\left( w_{+},w_{-}\right) =\displaystyle \int\limits_{\Omega
_{+}}g_{+}w_{+}dxdy+\displaystyle \int\limits_{\Omega _{-}}g_{-}w_{-}dxdy,
\end{array}
\right.
\end{equation*}
then
\begin{equation*}
a\left( \left( u_{+},u_{+}\right) ,\left( w_{+},w_{-}\right) \right)
+b\left( \left( u_{+},u_{-}\right) ,\left( w_{+},w_{-}\right) \right)
=l\left( w_{+},w_{-}\right) .
\end{equation*}
Now, taking $w_{+}$ in $\mathcal{D}\left( \Omega _{+}\right) $ and $w_{-}=0$
, we have in the sense of distributions 
\begin{equation*}
-\text{div}\left( d_{+}\nabla u_{+}\right) +r_{+}u_{+}=g_{+}\text{ in }
\Omega _{+}.
\end{equation*}
For $g_{+}$ in $L^{2}\left( \Omega _{+}\right) $, the trace of $\displaystyle d_{+}\frac{\partial u_{+}}{\partial \nu }$ on $\Gamma _{0}$, can be defined in the dual space $\left( \widetilde{H^{1/2}}\left( \Gamma _{0}\right) \right) ^{\prime}$ of $\widetilde{H^{1/2}}\left( \Gamma _{0}\right) $. In fact, operator $\partial /\partial \nu $ maps continuously from $H^{1}(0,1)$ into $L^{2}(0,1)$ and $L^{2}(0,1)$ into $H^{-1}(0,1)$, then, by interpolation it maps continuously from $\left( H^{1}(0,1);L^{2}(0,1)\right) _{1/2,2}$ into $\left(L^{2}(0,1),H^{-1}(0,1)\right) _{1/2,2}$; but we know that
\begin{equation*}
\left( H^{1}(0,1),L^{2}(0,1)\right) _{1/2,2}=H^{1/2}(0,1),
\end{equation*}
and
\begin{equation*}
\left( L^{2}(0,1),H^{-1}(0,1))\right) _{1/2,2}=\left[ \left(
H_{0}^{1}(0,1),L^{2}(0,1)\right) _{1/2,2}\right] ^{\prime } = \left( \widetilde{H^{1/2}}\left( a,b\right) \right) ^{\prime },
\end{equation*}
see \cite{Luc Tartar} p. 160.

The Green's formula for $w_{+}$ in $H_{\Gamma _{+}}^{1}\left( \Omega
_{+}\right) $ gives
\begin{eqnarray*}
\displaystyle \int_{\Omega _{+}}\left( d_{+}\nabla u_{+}.\nabla
w_{+}+r_{+}u_{+}w_{+}\right) dxdy &=& \displaystyle \int_{\Omega _{+}}\left( -\nabla .\left( d_{+}\nabla u_{+}\right)
w_{+}+r_{+}u_{+}w_{+}\right) dxdy \\
&&+\displaystyle \int_{\Gamma _{0}}\left( -d_{+}\frac{\partial u_{+}}{\partial \nu }\right) w_{+}dy,
\end{eqnarray*}
where the last integral means that
\begin{equation*}
\displaystyle \int_{\Gamma _{0}}\left( -d_{+}\frac{\partial u_{+}}{\partial \nu }\right) w_{+}dy := \left\langle d_{+}\frac{\partial u_{+}}{\partial \nu }
;w_{+}\right\rangle _{\left( \widetilde{H^{1/2}}\left( \Gamma _{0}\right)
\right) ^{^{\prime }}\times \left( \widetilde{H^{1/2}}\left( \Gamma
_{0}\right) \right) }.
\end{equation*}
Similarly, for 
\begin{equation*}
-\text{div}\left( d_{-}\nabla u_{-}\right) +r_{-}u_{-}=g_{-}\text{ in }\Omega _{-},
\end{equation*}
we obtain
\begin{eqnarray*}
\displaystyle \int_{\Omega _{-}}\left( d_{-}\nabla u_{-} \cdot \nabla w_{-} + r_{-}u_{-}w_{-}\right) dxdy &=&\displaystyle \int_{\Omega _{-}}\left( -\nabla \cdot \left( d_{-}\nabla u_{-}\right)
w_{-}+r_{-}u_{-}w_{-}\right) dxdy \\ 
&&+ \displaystyle \int_{\Gamma _{0}}\left( d_{-}\frac{
\partial u_{-}}{\partial \nu }\right) w_{-}dy,
\end{eqnarray*}
as above, the last integral means that
\begin{equation*}
\displaystyle \int_{\Gamma _{0}}\left( -d_{-}\frac{\partial u_{-}}{\partial \nu }\right)
w_{-}dy:=\left\langle d_{-}\frac{\partial u_{-}}{\partial \nu }
;w_{-}\right\rangle _{\left( \widetilde{H^{1/2}}\left( \Gamma _{0}\right)
\right) ^{^{\prime }}\times \left( \widetilde{H^{1/2}}\left( \Gamma
_{0}\right) \right) }.
\end{equation*}
It follows that
\begin{equation*}
\displaystyle \int_{\Gamma _{0}}\left( -d_{+}\frac{\partial u_{+}}{\partial \nu }\right)
w_{+}dy+\displaystyle \int_{\Gamma _{0}}\left( d_{-}\frac{\partial u_{-}}{\partial \nu }
\right) w_{-}dy+\displaystyle \int_{\Gamma _{0}}q\left( u_{+}-u_{-}\right) \left(
w_{+}-w_{-}\right) dy=0;
\end{equation*}
taking $w_{-}=0$, we deduce that 
\begin{equation*}
-d_{+}\frac{\partial u_{+}}{\partial \nu }+q\left( u_{+}-u_{-}\right) =0 \text{ in }\left( \widetilde{H^{1/2}}\left( \Gamma _{0}\right) \right)^{^{\prime }},
\end{equation*}
in the same way, $w_{+}=0$ gives
\begin{equation*}
d_{-}\frac{\partial u_{-}}{\partial \nu }-q\left( u_{+}-u_{-}\right) =0\text{ in }\left( \widetilde{H^{1/2}}\left( \Gamma _{0}\right) \right) ^{^{\prime}}.
\end{equation*}

\section{Recall on sectorial operators}\label{Sect 3}

Let $\omega \in \left[ 0,\pi \right]$. We put
\begin{equation}
S_{\omega }:=\left\{ 
\begin{array}{ll}
\left\{ z\in \mathbb{C} \backslash \left\{ 0\right\} :\left\vert \arg(z)\right\vert <\omega \right\} &
\text{if }\omega \in (0,\pi ], \medskip \\ 
(0,+\infty ) & \text{if }\omega =0.
\end{array}
\right.  \label{Scteur}
\end{equation}

Let us recall some known results from \cite{Haase}. 
\begin{definition}
Let $\omega \in \lbrack 0,\pi )$. A linear operator $\Lambda $ on a complex Banach space $E$ is called sectorial of angle $\omega $ if
\begin{enumerate}
\item $\sigma (\Lambda )\subset \overline{S_{\omega }}$ and

\item $M(\Lambda ,\omega ^{\prime }):=\underset{\lambda \in \mathbb{C} \backslash \overline{S_{\omega ^{\prime }}}}{\sup }\left\Vert \lambda (\Lambda -\lambda I)^{-1} \right\Vert <\infty $ \ for all $\omega' \in (\omega ,\pi )$.
\end{enumerate}

We then write: $\Lambda \in Sect(\omega )$. The following angle
\begin{equation*}
\omega _{\Lambda }:=\min \left\{ \omega \in \lbrack 0,\pi ):\Lambda \in
Sect(\omega )\right\} ,
\end{equation*}
is called the spectral angle of $\Lambda $. 
\end{definition}
We recall the following properties of the set $Sect(\omega )$. It is clear that Statement $2$.
implies necessarily that $\Lambda $ is closed.

\begin{proposition}\label{Prop sect}
If $\ (-\infty ,0)\subset \rho (\Lambda )$ and 
\begin{equation*}
M(\Lambda ):=M(\Lambda ,\pi ):=\underset{t>0}{\sup }\left\Vert t(\Lambda
+tI)^{-1}\right\Vert <\infty ,
\end{equation*}
then $M(\Lambda )\geqslant 1$ and 
\begin{equation*}
\Lambda \in Sect\left(\pi -\arcsin (1/M(\Lambda ))\right).
\end{equation*}
\end{proposition}

\begin{proposition}
Let $\Lambda \in Sect(\omega_\Lambda)$ and $\nu \in (0,1/2]$. Then $\Lambda
^{\nu }\in Sect(\nu \omega _{\Lambda}),$ and therefore $-\Lambda ^{\nu }$
generates an analytic semigroup.
\end{proposition}
For more details, see \cite{Haase}, p. 80-81.

\begin{definition}
We denote by BIP$(E,\theta)$ (see \cite{pruss-sohr}), the class of sectorial injective operators $T$, on the Banach space $E$, such that
\begin{itemize}
\item[] $i) \quad ~~\overline{D(T)} = \overline{R(T)} = E,$

\item[] $ii) \quad ~\forall~ s \in \mathbb{R}, \quad T^{is} \in \mathcal{L}(E),$

\item[] $iii) \quad \exists~ C \geq 0 ,~ \forall~ s \in \mathbb{R}, \quad ||T^{is}||_{\mathcal{L}(E)} \leq C e^{|s|\theta}$.
\end{itemize}
\end{definition}

\begin{definition}
A Banach space $E$ is a $UMD$ space if and only if for all $1<p<+\infty$ the Hilbert transform is continuous from $L^{p}(\mathbb{R};E)$ into
itself, see \cite{Bourgain} and \cite{Burkholder}.
\end{definition}
Now, let us recall some important result on the well-known functional calculus. To this end, we set 
\begin{equation*}
H^{\infty }(S_{\omega })=\left\{ f:f\text{ is an holomorphic and bounded
function on }S_{\omega }\right\} \text{,}
\end{equation*}
with $\omega \in (0,\pi )$; see for instance \cite{Haase}, p. 28.
\begin{definition}
Let $\Lambda$ be a closed linear densely defined operator in $E$. We say that $\Lambda$ has bounded $H^\infty(S_\omega)$ functional calculus if for every $f \in H^\infty(S_\omega)$ the operator $f(\Lambda)$ is bounded and there exists $C > 0$ (independent of $f$) such that 
$$\|f(\Lambda)\|_{\mathcal{L}(E)} \leqslant C \|f\|_\infty.$$
\end{definition}
\begin{proposition}\label{Prop Dore-Labbas calcul fonctionnel}
Let $\Lambda$ be an injective sectorial operator with dense range. If $f\in H^{\infty }(S_{\omega })$ is such that $1/f\in H^{\infty }(S_{\omega })$ and 
\begin{equation*}
(1/f)(\Lambda )\in \mathcal{L}(E),
\end{equation*}
then $f(\Lambda )$ is boundedly invertible and 
\begin{equation}
\left[ f(\Lambda )\right] ^{-1}=(1/f)(\Lambda ). \label{HinfiniCalcul}
\end{equation}
\end{proposition}
This result is proved in \cite{DoreLabbas} Proposition 3.3, p. 1873. For the definition of $f(\Lambda)$, see, for instance, section 3 on the functional calculus in \cite{DoreLabbas}, p. 1871-1872.

\section{Assumptions and main results}\label{Sect Assumptions and main results}

Let $A$ be a linear closed densely defined operator in a complex Banach space $E$ and assume in all this paper that
\begin{equation}\label{UMD}
E\text{ is a UMD space,} 
\end{equation}
\begin{equation}\label{0 dans rho de A}
0\in \rho(A)
\end{equation}
\begin{equation}\label{MoinsAbip}
\begin{array}{l}
-A \text{ is a sectorial operator and has bounded } H^\infty(S_\varepsilon) \\
\text{functional calculus for some fixed } \varepsilon \in (0,\pi/2).
\end{array}
\end{equation}
\begin{remark}
As a consequence, we have the two following results:
\begin{enumerate}
\item $-A \in$ BIP$(E,\varepsilon)$, see \cite{DoreLabbas}, p. 1876.

\item $\sqrt{-A}$ has bounded $H^\infty(S_{\varepsilon/2})$ functional calculus, see \cite{DoreLabbas}, Proposition 3.4, p. 1873.
\end{enumerate}
\end{remark}
\begin{remark}
In concrete examples, operator $-A$ represents, for instance, an elliptic operator set in some bounded regular domain in $L^p$-spaces.
\end{remark}
We define operator $\mathcal{S}$ by
$$\left\{\begin{array}{cll}
D\left( \mathcal{S}\right) &=&\left\{ 
\begin{array}{l}
w\in L^{p}(-\ell ,L;E):w_{-}\in W^{2,p}(-\ell ,0;E),w_{+}\in W^{2,p}(0,L;E)
\\ 
w_{-}(-\ell )=0,w_{+}(L)=0,d_{-}w_{-}^{\prime }(0)=q\left(
w_{+}(0)-w_{-}(0)\right) \text{ } \\ 
\text{and }d_{+}w_{+}^{\prime }(0)=q\left( w_{+}(0)-w_{-}(0)\right)
\end{array}
\right\}, \medskip \\
\left( \mathcal{S}w\right) (x) &=&\left\{ 
\begin{array}{ll}
d_{-}w_{-}''(x)+d_{-}Aw_{-}(x)-r_{-}w_{-}(x) & \text{in } (-\ell ,0), \\ 
d_{+}w_{+}''(x)+d_{+}Aw_{+}(x)-r_{+}w_{+}(x) & \text{in } (0,L) .
\end{array}\right.
\end{array}\right.$$

Thanks to the fact that the domain is cylindrical, we will give an explicit
expression of the resolvent operator of $\mathcal{S}$ by using essentially
the analytic semigroups theory and the functional calculus.

Therefore, our aim results are the following:

\begin{theorem}\label{Th général}
Assume that \eqref{UMD}, \eqref{0 dans rho de A} and \eqref{MoinsAbip} hold. Then, operator $\mathcal{S}$ generates an analytic semigroup in $L^{p}(-\ell ,L;E).$
\end{theorem}

As corollaries, we obtain.

\begin{theorem}
Operator $\mathcal{S}_{0}$ generates an analytic semigroup in $L^{p}(-\ell
,L;E_{0}).$
\end{theorem}

\begin{theorem}
Operator $\mathcal{P}$ generates an analytic semigroup in $L^{p}(\Omega ).$
\end{theorem}

\section{Preliminary results}\label{Sect 5}

In this section, we recall some useful results and we state some technical results.

\begin{proposition}\label{Prop arg}
Let $c \in \mathbb{R} \setminus \{0\}$ and $z\in \mathbb{C}\setminus \mathbb{R}$. Then, we have
$$\left\{\begin{array}{cccccccl}
0 &<& |\arg(z+c)| &<& |\arg(z)|   &<& \pi & \text{if } c > 0, \\ \ecart
0 &<& |\arg(z)|   &<& |\arg(z+c)| &<& \pi & \text{if } c < 0.
\end{array}\right.$$
\end{proposition}

\begin{proof}
For $z\in \mathbb{C}\setminus \mathbb{R}$, we have
\begin{equation}\label{Def arg}
\arg(z) = 2 \arctan\left(\frac{\text{Im}(z)}{\text{Re}(z) + |z|}\right),
\end{equation}
and 
\begin{enumerate}
\item if $c > 0$, then 
$$\begin{array}{lll}
\dis \left|\arg(z)\right| & = & \dis 2 \arctan \left(\frac{\left|\text{Im}(z+c)\right|}{\text{Re}(z) + |z+c-c|}\right) \\ \ecart
& > & \dis 2 \arctan\left(\frac{\left|\text{Im}(z+c)\right|}{\text{Re}(z)+c + |z+c|}\right) \\ \ecart
& = & \dis \left|\arg(z+c)\right|.
\end{array}$$

\item if $c < 0$, then 
$$\begin{array}{lll}
\dis\left|\arg(z)\right| & = &\dis 2 \arctan\left(\frac{\left|\text{Im}(z+c)\right|}{\text{Re}(z) + c - c + |z|}\right) \\ \ecart
& = & \dis 2 \arctan\left(\frac{\left|\text{Im}(z+c)\right|}{\text{Re}(z) + c + |c| + |z|}\right) \\ \ecart
& < & \dis 2 \arctan\left(\frac{\left|\text{Im}(z+c)\right|}{\text{Re}(z) + c + |z+c|}\right) \\ \ecart
& = & \dis \left|\arg(z+c)\right|.
\end{array}$$
\end{enumerate}
\end{proof}
\begin{proposition}\label{Prop ineg arg}
Let $z_1,z_2 \in \mathbb{C}\setminus\{0\}$. Assume that $z_1+z_2 \neq 0$ and $\left|\arg(z_1) - \arg(z_2)\right| \leqslant \pi$. Then, we have
$$\min\left(\arg(z_1), \arg(z_2)\right) \leqslant \arg(z_1+z_2) \leqslant \max\left(\arg(z_1),\arg(z_2)\right).$$
\end{proposition}
\begin{proof}
Without loss of generality, it suffices to consider only the case where $\arg(z_1) \leqslant \arg(z_2)$.

\begin{enumerate}
\item Assume that $\arg(z_2) = \pi$. 

If $0 < \arg(z_1) < \pi$, from Proposition \ref{Prop arg}, we obtain the expected inequalities. 

If $\arg(z_1) = \arg(z_2) = \pi$, then 
\begin{equation}\label{eq egalité arg}
\arg(z_1) = \arg(z_1+z_2) = \arg(z_2).
\end{equation}
If $\arg(z_1) = 0$, then the expected inequalities hold since we have
$$\arg(z_1+z_2) = \left\{\begin{array}{ll}
\arg(z_1) & \text{if } |z_1|>|z_2|, \\
\arg(z_2) & \text{if } |z_2|>|z_1|.
\end{array}\right.$$

\item Assume that $\arg(z_2) = -\pi$, then $\arg(z_1) = \arg(z_2) = -\pi$ and \eqref{eq egalité arg} holds.

\item Assume that $\arg(z_2) \in (-\pi,0]$. Then, $\arg(z_1) \in (-\pi,\arg(z_2)]$, Im$(z_2)\leqslant 0$ and Im$(z_1) \leqslant 0$. From \eqref{Def arg}, we have
$$\frac{\text{Im}(z_1)}{\text{Re}(z_1) + |z_1|} \leqslant \frac{\text{Im}(z_2)}{\text{Re}(z_2) + |z_2|}.$$
Hence
$$\text{Im}(z_1)\left(\text{Re}(z_2) + |z_2|\right) + \text{Im}(z_2) \left(\text{Re}(z_2) + |z_2|\right) \leqslant \text{Im}(z_2) \left(\text{Re}(z_1) + |z_1|\right) + \text{Im}(z_2) \left(\text{Re}(z_2) + |z_2|\right),$$
wich gives
$$\text{Im}(z_1+z_2)\left(\text{Re}(z_2) + |z_2|\right) \leqslant \text{Im}(z_2) \left(\text{Re}(z_1+z_2) + |z_1| + |z_2|\right).$$
Since Im($z_2) \leqslant 0$, then we have 
$$\text{Im}(z_2)\left(|z_1|+|z_2|\right) \leqslant \text{Im}(z_2)|z_1+z_2|.$$
Therefore
$$\text{Im}(z_1+z_2)\left(\text{Re}(z_2) + |z_2|\right) \leqslant \text{Im}(z_2) \left(\text{Re}(z_1+z_2) + |z_1 + z_2|\right),$$
and due to \eqref{Def arg}, we obtain 
$$\arg(z_1+z_2) \leqslant \arg(z_2).$$
We extend the above result, by rotation, for $\arg(z_2) \in (-\pi,\pi)$.

\item Similarly, for $\arg(z_1)$, $\arg(z_2) \in [0,\pi)$, following the same steps, we deduce that
$$\arg(z_1) \leqslant \arg(z_1+z_2),$$
and, by rotation, we obtain the expected result.
\end{enumerate}
\end{proof}

\begin{proposition}\label{Prop DoreLabbas}
Let $z_1,z_2 \in \mathbb{C}\setminus\{0\}$. We have
$$|z_1+z_2| \geqslant \left(|z_1|+|z_2|\right)\left|\cos\left(\frac{\arg(z_1)-\arg(z_2)}{2}\right)\right|.$$
\end{proposition}
This result is given by Proposition 4.9, p. 1879 in \cite{DoreLabbas}.

\begin{proposition}\label{Prop DoreLabbas 2}
Let $0 < \alpha < \pi/2$ and $z \in S_\alpha$. We have
\begin{enumerate}
\item $\left|\arg\left(1 - e^{-z}\right) - \arg\left(1 + e^{-z}\right)\right| < \alpha.$

\item $\left|1 + e^{-z} \right| \geqslant 1 - e^{-\pi/(2 \tan(\alpha))}.$

\item $\dfrac{|z| \cos(\alpha)}{1+|z| \cos(\alpha)} \leqslant |1 - e^{-z}| \leqslant \dfrac{2|z|}{1+|z|\cos(\alpha)}.$
\end{enumerate}
\end{proposition}
This result is given in Proposition 4.10, p. 1880 in \cite{DoreLabbas}.

\begin{corollary}\label{Cor DoreLabbas 1}
Let $\theta_0, \theta_1 \in [0,\pi/2)$ with $\theta_0 < \theta_1$ and $\overline{L} >0$. Then, there exists $C>0$ such that for all $z\in S_{\theta_0}$ and all $\mu \in S_{\pi - \theta_1}\cup\{0\}$, we have 
$$\left|\frac{1-e^{-\overline{L}\sqrt{z + \mu}}}{\sqrt{z+\mu}\left(1+e^{-\overline{L}\sqrt{z+\mu}}\right)}\right| \leqslant \frac{C}{\sqrt{|z|+|\mu|}}.$$
\end{corollary}

\begin{proof}
 If $-\theta_1 < \arg(\mu) <\pi-\theta_1$, since $|\arg(z)| \leqslant \theta_0$, then from \refP{Prop ineg arg}, we have
$$-\theta_1 < \arg(z + \mu) < \pi - \theta_1,$$
and
$$-\frac{\theta_1}{2} < \arg(\overline{L}\sqrt{z+\mu}) = \frac{\arg(z + \mu)}{2} < \frac{\pi - \theta_1}{2}.$$
If $-\pi+\theta_1 < \arg(\mu) < \theta_1$, since $|\arg(z)| \leqslant \theta_0$, then from \refP{Prop ineg arg}, we have
$$-\pi+\theta_1 < \arg(z + \mu) < \theta_1,$$
and
$$\frac{-\pi + \theta_1}{2} < \arg(\overline{L}\sqrt{z+\mu}) = \frac{\arg(z + \mu)}{2} < \frac{\theta_1}{2}.$$
Therefore, we always have
$$\overline{L}\sqrt{z+\mu} \in S_{\frac{\pi}{2} - \frac{\theta_1}{2}}.$$
From \refP{Prop DoreLabbas 2}, there exists $C > 0$ such that
$$\begin{array}{lll}
\dis\left|\frac{1-e^{-\overline{L}\sqrt{z + \mu}}}{\sqrt{z+\mu}\,(1+e^{-\overline{L}\sqrt{z+\mu}})}\right| &\leqslant & \dis \frac{2\overline{L}|\sqrt{z+\mu}|}{\left(1 + \overline{L}|\sqrt{z+\mu}| \cos\left(\frac{\pi}{2} - \frac{\theta_1}{2}\right)\right)|\sqrt{z+\mu}|\left(1 - e^{-\frac{\pi}{2\tan\left(\frac{\pi}{2} - \frac{\theta_1}{2}\right)}}\right)} \\ \\

& \leqslant & \frac{2\overline{L}}{\left(1 + \overline{L}|\sqrt{z+\mu}| \sin\left(\frac{\theta_1}{2}\right)\right)\left(1 - e^{-\frac{\pi}{2}\tan\left(\frac{\theta_1}{2}\right)}\right)} \\ \\

& \leqslant & \dis \frac{C}{|\sqrt{z+\mu}|} = \frac{C}{\sqrt{|z + \mu|}}. 
\end{array}$$
Moreover, from \refP{Prop DoreLabbas}, we have
$$\left|\frac{1-e^{-\overline{L}\sqrt{z + \mu}}}{\sqrt{z+\mu}\,(1+e^{-\overline{L}\sqrt{z+\mu}})}\right| \leqslant \frac{C}{\sqrt{\left(|z| + |\mu|\right)\left|\cos\left(\dfrac{\arg(z) - \arg(\mu)}{2}\right)\right|}}.$$ 
Since $|\arg(\mu)| < \pi-\theta_1$ and $|\arg(z)| \leqslant \theta_0$, with $\theta_1 > \theta_0$, it follows that
$$|\arg(z) - \arg(\mu)| \leqslant \pi - (\theta_1 - \theta_0) < \pi,$$
and
$$\cos\left(\frac{\arg(z) - \arg(\mu)}{2}\right) \geqslant \cos\left( \frac{\pi}{2} - \frac{\theta_1 - \theta_0}{2}\right) = \sin\left(\frac{\theta_1 - \theta_0}{2}\right) > 0,$$
which gives the result.
\end{proof}

\begin{corollary}\label{Cor DoreLabbas 2}
Let $\alpha \in (0,\pi/2]$, $\beta \in [0, \alpha/2]$ and $z \in \mathbb{C}\setminus \{0\}$  such that $|\text{Im}(z)| \leqslant \pi$. Then
\begin{enumerate}
\item if $-\beta \leqslant \arg(z) < \alpha-\beta$, then we have
$$- \beta \leqslant \arg\left(1 - e^{-z}\right) - \arg\left(1 + e^{-z}\right) < \alpha - \beta.$$

\item if $-\alpha+\beta < \arg(z) \leqslant \beta$, then we have
$$-\alpha + \beta < \arg\left(1 - e^{-z}\right) - \arg\left(1 + e^{-z}\right) \leqslant \beta.$$
\end{enumerate}
 
\end{corollary}

\begin{proof}
Since $S_\alpha$ is an open sector, then Proposition \ref{Prop DoreLabbas 2} remains true for $\alpha = \pi/2$.

\begin{enumerate}
\item First, let $\beta = 0$. Then, we have $0 \leqslant \arg(z) < \alpha$. As in the proof of statement 1 of Proposition~4.10, p. 1880 in \cite{DoreLabbas}, we have Re$(1-e^{-z})$, Re$(1+e^{-z}) > 0$. Let us prove that
$$\arg\left(1 - e^{-z}\right) - \arg\left(1 + e^{-z}\right) \geqslant 0.$$
To this end, we must show that
$$\arctan\left(\frac{\text{Im}(1 - e^{-z})}{\text{Re}(1 - e^{-z})}\right) \geqslant \arctan \left(\frac{\text{Im}(1 + e^{-z})}{\text{Re}(1 + e^{-z})}\right),$$
that is
$$\frac{\text{Im}(1 - e^{-z})}{\text{Re}(1 - e^{-z})} \geqslant \frac{\text{Im}(1 + e^{-z})}{\text{Re}(1 + e^{-z})},$$
or
\begin{equation}\label{eq Im Re}
\text{Im}(1 - e^{-z})\text{Re}(1 + e^{-z}) \geqslant \text{Im}(1 + e^{-z})\text{Re}(1 - e^{-z}).
\end{equation}
Since we have
$$\left\{\begin{array}{l}
\text{Re}(1-e^{-z}) = 1 - e^{-\text{Re}(z)}\cos(\text{Im}(z)), \\
\text{Re}(1+e^{-z}) = 1 + e^{-\text{Re}(z)}\cos(\text{Im}(z)), \\
\text{Im}(1-e^{-z}) =   e^{-\text{Re}(z)}\sin(\text{Im}(z)), \\
\text{Im}(1+e^{-z}) = - e^{-\text{Re}(z)}\sin(\text{Im}(z)),
\end{array}\right.$$
then, \eqref{eq Im Re} is equivalent to
$$e^{-\text{Re}(z)}\sin(\text{Im}(z))\left(1 + e^{-\text{Re}(z)}\cos(\text{Im}(z))\right) \geqslant - e^{-\text{Re}(z)}\sin(\text{Im}(z))\left(1 - e^{-\text{Re}(z)}\cos(\text{Im}(z)) \right);$$
hence
$$\sin(\text{Im}(z)) \geqslant 0.$$
which is true since $0 \leqslant \text{Im}(z)\leqslant \pi$. Now, taking into account that $0 \leqslant \arg(z) < \alpha$ and Proposition \ref{Prop DoreLabbas 2}, we obtain
\begin{equation}\label{ineq DL}
0 \leqslant \arg\left(1 - e^{-z}\right) - \arg\left(1 + e^{-z}\right) < \alpha.
\end{equation}
Now, let $\beta \in \left(0, \dfrac{\alpha}{2}\right]$ and $-\beta \leqslant \arg(z) < \alpha - \beta$. 

If $0 \leqslant \arg(z) < \alpha - \beta$, then from \eqref{ineq DL}, we deduce that
$$-\beta \leqslant 0 \leqslant \arg\left(1 - e^{-z}\right) - \arg\left(1 + e^{-z}\right) < \alpha - \beta.$$

If $-\beta < \arg(z) < 0$, then $z \in S_\beta$ and from Proposition \ref{Prop DoreLabbas 2}, we have
$$-\beta \leqslant \arg\left(1 - e^{-z}\right) - \arg\left(1 + e^{-z}\right) < \beta.$$
Note that, when $\arg(z) = -\beta$, then the previous inequality holds true since $\beta<\alpha$.
Finally, the result follows since $\beta \leqslant \dfrac{\alpha}{2}$.

\item In the same way, we obtain the expected result. 
\end{enumerate}

\end{proof}

\section{Proof of Theorem \ref{Th général}}\label{Sect 6}

\subsection{Spectral study of $\mathcal{S}$}

In this sections we will focus ourselves to study the spectral equation
\begin{equation}\label{SpectralEquation}
\mathcal{S}w-\lambda w=f\in \mathcal{E}=L^{p}(-\ell ,L;E),
\end{equation}
where $p\in (1,+\infty)$.

Our aim is to estimate the resolvent operator 
\begin{equation*}
\left\Vert \left( \mathcal{S}-\lambda I\right) ^{-1}\right\Vert _{\mathcal{L}\left(
L^{p}(-\ell ,L;E)\right) },
\end{equation*}
where $\lambda $ is a complex number in some sector to specify. This
estimate will allow us to prove that $\mathcal{S}$ generates an analytic
semigroup in $\mathcal{E}$. So, after the resolution of the spectral
equation, we have to estimate $\left\Vert w\right\Vert _{L^{p}(-\ell ,L;E)}$ that is 
\begin{equation*}
\left\Vert w_{-}\right\Vert_{L^{p}(-\ell ,0;E)} \quad \text{and} \quad \left\Vert w_{+}\right\Vert _{L^{p}(0,L;E)}.
\end{equation*}
We recall that all the constants $r_{-},r_{+},d_{-},d_{+},q$ are strictly
positive. In the sequel, we will use the following notations:
\begin{equation}\label{lambda+- et rho+-}
\lambda_\pm = \frac{\lambda}{d_\pm}, \quad \rho_\pm = \frac{r_\pm}{d_\pm}, \quad q_\pm = \frac{q}{d_\pm} \quad \text{and} \quad g_\pm = \frac{f_\pm}{d_\pm}.
\end{equation}

\subsubsection{The system verified by the spectral equation}

Equation \eqref{SpectralEquation} can be formulated as
\begin{equation*}
(P_{A})\left\{ 
\begin{array}{l}
\left\{ 
\begin{array}{ll}
w''_{-}(x) + \left(A - \rho_- I - \lambda_- I \right) w_{-}(x) = g_{-}(x) & \text{in }(-\ell ,0),\medskip \\ 
w''_{+}(x) + \left(A - \rho_+ I - \lambda_+ I\right)w_{+}(x) = g_{+}(x) & \text{in }(0,L),
\end{array}
\right. \medskip \\ 
\left\{ 
\begin{array}{l}
w_{-}(-\ell )=0,\text{\medskip } \\ 
w_{+}(L)=0,
\end{array}
\right. \medskip \\ 
\left\{ \begin{array}{l}
w'_{-}(0) = q_- \left( w_{+}(0)-w_{-}(0)\right), \\ 
w'_{+}(0) = q_+ \left( w_{+}(0)-w_{-}(0)\right).
\end{array}
\right.
\end{array}
\right.
\end{equation*}

Assume that the complex $\lambda $ satisfies:
\begin{equation}
\left\vert \arg (\lambda )\right\vert <\pi -\varepsilon,
\label{ConditionOnLambda}
\end{equation}
where $\varepsilon \in (0,\pi)$ as in \eqref{MoinsAbip}. Set
\begin{equation*}
A_{-} = A - \rho_- I - \lambda_- I \quad \text{and} \quad A_{+} = A - \rho_+ I - \lambda_+ I,
\end{equation*}
so we have 
\begin{equation*}
D(A_{-})=D(A_{+})=D(A).
\end{equation*}

\begin{proposition}
Operators $-A_-$ and $-A_+$ are sectorial and satisfy  
\begin{equation*}
\left\{\begin{array}{l}
-A_{-}\in Sect\left( \pi -\arcsin \left( 1/M(-A_{-})\right) \right), \\
-A_{+}\in Sect\left( \pi -\arcsin \left( 1/M(-A_{+})\right) \right),
\end{array}\right.
\end{equation*}
where $M(-A_\pm)$ is defined in Proposition \ref{Prop sect}.
\end{proposition}
\begin{proof}
If $\lambda \in \mathbb{R}_+$, then due to \eqref{MoinsAbip}, $-A_-$ and $-A_+$ are sectorial operators. 

Now, let $\lambda \in S_{\pi-\varepsilon}\setminus\mathbb{R_+}$ and $t>0$. We will verify that $-A_-$ is a sectorial operator in $E$ and $(-\infty ,0]\subset \rho (-A_{-})$. Since $\rho_- + t > 0$, from Proposition \ref{Prop arg}, we obtain 
$$\left|\arg\left(\rho_- + t + \lambda_-\right)\right| < \left|\arg\left(\lambda_-\right)\right| = |\arg(\lambda)| < \pi - \varepsilon.$$
Then
\begin{equation*}
M(-A_{-})=\underset{t>0}{\sup }\left\|
t(-A_{-}+tI)^{-1}\right\| = \underset{t>0}{\sup }\,t \left\|\left(-A + \left(\rho_- + t + \lambda_-\right) I\right)^{-1}\right\|_{\mathcal{L}(E)}.
\end{equation*}
Then, from \eqref{MoinsAbip} and Proposition \ref{Prop DoreLabbas}, we have
$$\begin{array}{lll}
\displaystyle\left\|\left(-A + \left(\rho_- + t + \lambda_-\right) I\right)^{-1}\right\|_{\mathcal{L}(E)} & \leqslant & \displaystyle\frac{C_A}{1 + \left|\rho_- + t + \lambda_-\right|} \leqslant \displaystyle\frac{C_A}{\left|\rho_- + t + \lambda_-\right|} \medskip \\
& \leqslant & \displaystyle\frac{C_A}{\left(\left|\rho_- + \lambda_-\right| + t\right)\cos\left(\frac{\arg\left(\rho_- + \lambda_-\right)}{2}\right)} \medskip \\
& \leqslant & \displaystyle\frac{C_A}{t\cos\left(\frac{\left|\arg\left(\rho_- + \lambda_-\right)\right|}{2}\right)}.
\end{array}$$
Moreover, from Proposition \ref{Prop arg}, we obtain 
$$\left|\arg\left(\left|\rho_- + \lambda_-\right|\right)\right| < \left|\arg\left(\lambda_-\right)\right| = |\arg(\lambda)| < \pi-\varepsilon,$$
and thus
$$\cos\left(\frac{\left|\arg\left(\rho_- + \lambda_-\right)\right|}{2}\right) > \cos\left(\frac{\pi}{2}-\frac{\varepsilon}{2}\right) = \sin\left(\frac{\varepsilon}{2}\right) > 0.$$
Finally, we have
\begin{equation*}
M(-A_{-})\leqslant \underset{t>0}{\sup }\left( \dfrac{t \,C_A}{t \sin\left(\dfrac{\varepsilon}{2}\right)}\right) = \dfrac{C_A}{\sin\left(\dfrac{\varepsilon}{2}\right)} < + \infty.
\end{equation*}
Hence, due to Proposition \ref{Prop sect}, we obtain the expected result for $-A_-$. For $-A_+$, the proof is similar.
\end{proof}

Therefore, the following operators
\begin{equation*}
B = - \sqrt{- A}, \quad B_{-} =- \sqrt{ - A_-} \quad \text{and} \quad B_{+} = - \sqrt{- A_+},
\end{equation*}
are well defined for all $\lambda \in S_{\pi-\varepsilon} \cup \{0\}$ and generate analytic semigroups in $E$, see \cite{Balakrishnan}.

By using estimates (28)-(29) in Lemma 4.2, see \cite{FaviniLabbasMaingotTorel}, there exists $C >0$, independent of $\lambda $ such that for all $z \in S_{\pi-\varepsilon} \cup \{0\}$, we have
\begin{equation}\label{IndependenceOfLambda}
\left\{ 
\begin{array}{l}
\left\Vert \left( B_{-}-zI\right) ^{-1}\right\Vert_{\mathcal{L}(E)} \leqslant \dfrac{C}{\sqrt{1+ |\lambda_- + \rho_-|}+\left\vert z\right\vert } \displaystyle\leqslant \frac{C}{1 + |z|}, \medskip
\\ 
\left\Vert \left( B_{+}-zI\right) ^{-1}\right\Vert_{\mathcal{L}(E)} \leqslant \dfrac{C}{\sqrt{1 + |\lambda_+ + \rho_+|}+\left\vert z\right\vert } \displaystyle\leqslant \frac{C}{1 + |z|}.
\end{array}\right. 
\end{equation}
Note that these estimates do not depend on $\lambda$. It follows that
\begin{equation}\label{estim B+-}
\left\|B_-^{-1}\right\|_{\mathcal{L}(E)} \leqslant \frac{C \sqrt{d_-}}{\sqrt{d_- + |\lambda + r_-|}} \quad \text{and} \quad \left\|B_+^{-1}\right\|_{\mathcal{L}(E)} \leqslant \frac{C \sqrt{d_+}}{\sqrt{d_+ + |\lambda + r_+|}}.
\end{equation}
Our problem $\left( P_{A}\right) $ can be written in the following form
\begin{equation}\label{PA}
\left\{ 
\begin{array}{l}
\left\{ 
\begin{array}{ll}
w''_{-}\left( x\right) -B_{-}^{2}w_{-}\left( x\right)=g_{-}(x) & \text{in }(-\ell ,0), \\ 
w_{+}''\left( x\right) -B_{+}^{2}w_{+}\left( x\right)=g_{+}(x) & \text{in }(0,L),
\end{array}
\right. \medskip \\ 
\left\{ 
\begin{array}{l}
w_{-}(-\ell )=0, \\ 
w_{+}(L)=0,
\end{array}
\right. \medskip \\ 
\left\{ 
\begin{array}{l}
w'_{-}(0)=q_-\left( w_{+}(0)-w_{-}(0)\right), \\ 
w'_{+}(0)=q_+\left( w_{+}(0)-w_{-}(0)\right).
\end{array}
\right.
\end{array}
\right.
\end{equation}
Then, we have 
$$\left\{\begin{array}{llll}
w_{-}\left( x\right) &=& e^{(x+\ell)B_{-}}j_{-}+e^{-xB_{-}}k_{-}+v_{-}(g_{-})(x), & x\in (-\ell ,0), \medskip \\
w_{+}\left( x\right) &=& e^{xB_{+}}j_{+}+e^{(L-x)B_{+}}k_{+}+v_{+}(g_{+})(x), & x\in (0,L),
\end{array}\right.$$
with $j_{\pm },$ $k_{\pm }\in E$ and
\begin{equation}\label{v_pm(g_pm) (x)}
\left\{\begin{array}{lll}
v_{-}(g_{-})(x) &=&\displaystyle \frac{1}{2}B_{-}^{-1} \int_{-\ell}^{x}e^{(x-t)B_{-}}g_{-}(t)dt+\frac{1}{2}B_{-}^{-1} \int_{x}^{0}e^{(t-x)B_{-}}g_{-}(t)dt, \medskip \\
v_{+}(g_{+})(x) &=& \displaystyle \frac{1}{2}B_{+}^{-1} \int_{0}^{x}e^{(x-t)B_{+}}g_{+}(t)dt+\frac{1}{2}B_{+}^{-1} \int_{x}^{L}e^{(t-x)B_{+}}g_{+}(t)dt;
\end{array}\right.
\end{equation}

see for instance Proposition 4.4, p. 1878 in \cite{DoreLabbas}.

Then, we deduce that
$$\left\{\begin{array}{lll}
w_{-}'\left( 0\right) &=& B_{-}e^{\ell B_{-}}j_{-}-B_{-}k_{-} + v_{-}'(g_{-})(0), \medskip \\
w_{+}'\left( 0\right)&=&B_{+}j_{+}-B_{+}e^{LB_{+}}k_{+} + v_{+}'(g_{+})(0),
\end{array}\right.$$
where
$$
v_{-}'(g_{-})(0) = \frac{1}{2} \int_{-\ell}^{0}e^{-tB_{-}}g_{-}(t)dt \quad \text{and} \quad v_{+}'(g_{+})(0) = - \dfrac{1}{2}\displaystyle \int_{0}^{L}e^{tB_{+}}g_{+}(t)dt.
$$
The boundary and the interface conditions give
\begin{equation}\label{syst Prop reg}
\left\{ 
\begin{array}{l}
j_{-}=-e^{\ell B_{-}}k_{-}-v_{-}(g_{-})(-\ell ), \\ 
k_{+}=-e^{LB_{+}}j_{+}-v_{+}(g_{+})(L), \\ 
\\ 
d_{-}\left( B_{-}e^{\ell B_{-}}j_{-}-B_{-}k_{-}+v_{-}'(g_{-})(0)\right) =\medskip \\ 
q\left[ \left( j_{+}+e^{LB_{+}}k_{+}+v_{+}(g_{+})(0)\right) -\left( e^{\ell
B_{-}}j_{-}+k_{-}+v_{-}(g_{-})(0)\right) \right], \\ 
\\ 
d_{+}\left( B_{+}j_{+}-B_{+}e^{LB_{+}}k_{+}+v_{+}'(g_{+})(0)\right)
=\medskip \\ 
q\left[ \left( j_{+}+e^{LB_{+}}k_{+}+v_{+}(g_{+})(0)\right) -\left( e^{\ell
B_{-}}j_{-}+k_{-}+v_{-}(g_{-})(0)\right) \right] ;
\end{array}
\right.
\end{equation}
then, the two last equations lead us to the following system
\begin{equation*}
\left\{ 
\begin{array}{l}
d_{-}\left[ B_{-}e^{\ell B_{-}}\left( -e^{\ell
B_{-}}k_{-}-v_{-}(g_{-})(-\ell )\right) -B_{-}k_{-}+v_{-}^{\prime }(g_{-})(0)
\right] \medskip \\ 
=q[\left( j_{+}+e^{LB_{+}}\left( -e^{LB_{+}}j_{+}-v_{+}(g_{+})(L)\right)
+v_{+}(g_{+})(0)\right) \medskip \\ 
\text{ \ }-e^{\ell B_{-}}\left( -e^{\ell B_{-}}k_{-}-v_{-}(g_{-})(-\ell
)\right) -k_{-}-v_{-}(g_{-})(0)], \\ 
\\ 
d_{+}\left[ B_{+}j_{+}-B_{+}e^{LB_{+}}\left(
-e^{LB_{+}}j_{+}-v_{+}(g_{+})(L)\right) +v_{+}^{\prime }(g_{+})(0)\right]
\medskip \\ 
=q[\left( j_{+}+e^{LB_{+}}\left( -e^{LB_{+}}j_{+}-v_{+}(g_{+})(L)\right)
+v_{+}(g_{+})(0)\right) \medskip \\ 
\text{ \ }-e^{\ell B_{-}}\left( -e^{\ell B_{-}}k_{-}-v_{-}(g_{-})(-\ell
)\right) -k_{-}-v_{-}(g_{-})(0),
\end{array}
\right.
\end{equation*}
which is equivalent to
\begin{equation*}
\left\{ 
\begin{array}{l}
-B_{-}e^{2\ell B_{-}}k_{-}-B_{-}e^{\ell B_{-}}v_{-}(g_{-})(-\ell
)-B_{-}k_{-}+v_{-}^{\prime }(g_{-})(0)\medskip \\ 
=q_{-}j_{+}-q_{-}e^{2LB_{+}}j_{+}-q_{-}e^{LB_{+}}v_{+}(g_{+})(L)+q_{-}v_{+}(g_{+})(0)\medskip
\\ 
\text{ \ }+q_{-}e^{2\ell B_{-}}k_{-}+q_{-}e^{\ell B_{-}}v_{-}(g_{-})(-\ell
)-q_{-}k_{-}-q_{-}v_{-}(g_{-})(0), \\ 
\\ 
B_{+}j_{+}+B_{+}e^{2LB_{+}}j_{+}+B_{+}e^{LB_{+}}v_{+}(g_{+})(L)+v_{+}^{
\prime }(g_{+})(0)\medskip \\ 
=q_{+}j_{+}-q_{+}e^{2LB_{+}}j_{+}-q_{+}e^{LB_{+}}v_{+}(g_{+})(L)+q_{+}v_{+}(g_{+})(0)\medskip
\\ 
\text{ \ }+q_{+}e^{2\ell B_{-}}k_{-}+q_{+}e^{\ell B_{-}}v_{-}(g_{-})(-\ell
)-q_{+}k_{-}-q_{+}v_{-}(g_{-})(0).
\end{array}\right.
\end{equation*}
Therefore, the above system becomes
\begin{equation*}
\left\{ 
\begin{array}{l}
\left[ B_{-}\left( I+e^{2\ell B_{-}}\right) -q_{-}\left( I-e^{2\ell
B_{-}}\right) \right] k_{-}+q_{-}\left( I-e^{2LB_{+}}\right) j_{+}=(\Pi'), \\ 
\\ 
q_{+}\left( I-e^{2\ell B_{-}}\right) k_{-}+\left[ B_{+}\left(
I+e^{2LB_{+}}\right) -q_{+}\left( I-e^{2LB_{+}}\right) \right] j_{+}=(\Pi''),
\end{array}
\right.
\end{equation*}
where
$$\left\{\begin{array}{lll}
(\Pi') & = & v_{-}'(g_{-})(0)-B_{-}e^{\ell B_{-}}v_{-}(g_{-})(-\ell) + q_{-}e^{LB_{+}}v_{+}(g_{+})(L)-q_{-}v_{+}(g_{+})(0) \medskip \\
&&- q_{-}e^{\ell B_{-}}v_{-}(g_{-})(-\ell )+q_{-}v_{-}(g_{-})(0), \\ \\

(\Pi'') & = & -v_{+}'(g_{+})(0)-B_{+}e^{LB_{+}}v_{+}(g_{+})(L)-q_{+}e^{LB_{+}}v_{+}(g_{+})(L)+q_{+}v_{+}(g_{+})(0) \medskip \\ 
&&+q_{+}e^{\ell B_{-}}v_{-}(g_{-})(-\ell )-q_{+}v_{-}(g_{-})(0).
\end{array}\right.$$
It follows 
\begin{equation*}
\left\{ 
\begin{array}{l}
\left[ \left( I+e^{2\ell B_{-}}\right) -q_{-}B_{-}^{-1}\left( I-e^{2\ell
B_{-}}\right) \right] k_{-}+q_{-}B_{-}^{-1}\left( I-e^{2LB_{+}}\right)
j_{+}=B_{-}^{-1}(\Pi'), \\ 
\\ 
q_{+}B_{+}^{-1}\left( I-e^{2\ell B_{-}}\right) k_{-}+\left[ \left(
I+e^{2LB_{+}}\right) -q_{+}B_{+}^{-1}\left( I-e^{2LB_{+}}\right) \right]
j_{+}=B_{+}^{-1}(\Pi'').
\end{array}
\right.
\end{equation*}
The abstract determinant of this system is
\begin{eqnarray*}
D &=&\left[ \left( I+e^{2\ell B_{-}}\right) -q_{-}B_{-}^{-1}\left(
I-e^{2\ell B_{-}}\right) \right] \left[ \left( I+e^{2LB_{+}}\right)
-q_{+}B_{+}^{-1}\left( I-e^{2LB_{+}}\right) \right] \medskip \\
&&-q_{+}B_{+}^{-1}\left( I-e^{2\ell B_{-}}\right) q_{-}B_{-}^{-1}\left(
I-e^{2LB_{+}}\right) \medskip \\
&=& \left( I+e^{2\ell B_{-}}\right) \left( I+e^{2LB_{+}}\right) - \left(I+e^{2\ell B_{-}}\right) q_{+}B_{+}^{-1}\left( I-e^{2LB_{+}}\right) \medskip
\\
&& - q_{-}B_{-}^{-1}\left( I-e^{2\ell B_{-}}\right) \left(
I+e^{2LB_{+}}\right).
\end{eqnarray*}

\subsubsection{Invertibility of $I+e^{2\ell B_-}$ and $I+e^{2L B_+}$}

Now let us study, for instance, the invertibility of $I+e^{2L B_+}$, the same method can be used for $I+e^{2\ell B_-}$. For a fixed $\lambda \in S_{\pi-\varepsilon}$, operator $-A_+ = -A + (\rho_+ + \lambda_+)I$ has bounded $H^\infty(S_\varepsilon)$ functional calculus in virtue of Corollary 5.5.5, p. 122 in \cite{Haase}; then using Proposition 3.4, p. 1873 in \cite{DoreLabbas}, $\sqrt{-A_+}$ has bounded $H^\infty(S_{\varepsilon/2})$ functional calculus.

Now, let $z \in S_{\varepsilon/2}$, then by \refP{Prop DoreLabbas 2}, we have
$$\left|1 + e^{-2Lz}\right| \geqslant 1 - e^{-\frac{\pi}{2 \tan(\varepsilon/2)}} > 0.$$
Using the same reasoning as in \cite{DoreLabbas}, p. 1883, we consider the following function
$$\begin{array}{cccl}
f : & S_{\varepsilon/2} & \longrightarrow & \mathbb{C} \medskip \\ 
 & z & \longmapsto & 1 + e^{-2Lz},
\end{array}$$
which does not vanish on $S_{\varepsilon/2}$ and $1/f$ belongs to $H^\infty(S_{\varepsilon/2})$ with norm bounded. Therefore
$$f(\sqrt{-A_+}) = I + e^{-2L \sqrt{-A_+}} = I + e^{2L B_+},$$
is invertible with bounded inverse 
$$\left(I + e^{2L B_+}\right)^{-1} = \left(\frac{1}{f}\right) (\sqrt{-A_+}),$$
with norm independent of $\lambda$. We have also used \refP{Prop Dore-Labbas calcul fonctionnel}. We then conclude by the following lemma
\begin{lemma}\label{Inv (I-exp2Q)}
There exists a constant $C > 0$ independent of $\lambda$, such that operators $I + e^{2\ell B_-}$ and $I + e^{2L B_+}$ are boundedly invertible and 
\begin{equation*}
\left\Vert \left( I + e^{2\ell B_-}\right) ^{-1}\right \Vert_{\mathcal{L}(E)} \leqslant C \quad \text{and} \quad \left\Vert \left( I + e^{2L B_+}\right) ^{-1}\right \Vert_{\mathcal{L}(E)}\leqslant C.  
\end{equation*}
\end{lemma}
Therefore, we can write
$$D = \left( I+e^{2\ell B_{-}}\right) \left( I+e^{2LB_{+}}\right) D_{\ast },$$
where
\begin{equation*}
D_{\ast }=\left[ I-q_{+}B_{+}^{-1}\left( I-e^{2LB_{+}}\right) \left(
I+e^{2LB_{+}}\right) ^{-1}-q_{-}B_{-}^{-1}\left( I-e^{2\ell B_{-}}\right)
\left( I+e^{2\ell B_{-}}\right) ^{-1}\right] .
\end{equation*}

\subsubsection{Invertibility of the determinant}

In order to invert $D_{\ast},$ we will also use the $H^{\infty }$-calculus for sectorial operators. To this end, we consider the following function
\begin{equation*}
f(z)=\left[ 1+\frac{q_{+}\left( 1-e^{-2L\sqrt{z+\lambda _{+}+\rho _{+}}}\right) }{\sqrt{z+\lambda _{+}+\rho _{+}}\left( 1+e^{-2L\sqrt{z+\lambda_{+}+\rho _{+}}}\right) }+\frac{q_{-}\left( 1-e^{-2\ell \sqrt{z+\lambda_{-}+\rho _{-}}}\right) }{\sqrt{z+\lambda _{-}+\rho _{-}}\left( 1+e^{-2\ell \sqrt{z+\lambda _{-}+\rho _{-}}}\right) }\right],
\end{equation*}
for $\lambda \in S_{\pi - \varepsilon_0} \cup \{0\}$ and for all $z \in S_{\varepsilon}$, where $\varepsilon$ is defined in \eqref{MoinsAbip} and $ \varepsilon_0$ is fixed such that
$$\varepsilon < \varepsilon_0 <\frac{\pi}{2}.$$ 
Recall that $\rho_\pm$, $q_\pm$ and $\lambda_\pm$ are defined in \eqref{lambda+- et rho+-}. 

\begin{proposition}\label{Prop f>0}
There exists $R>0$ such that for all $z\in S_{\varepsilon}$ and $\lambda \in S_{\pi - \varepsilon_0} \setminus B(0,R)$, where $B(0,R)$ is the ball with center $0$ and radius $R$, we have
$$\left|f(z)\right| > \frac{\sqrt{2}}{2}.$$
\end{proposition}

\begin{proof}
Let $z \in S_{\varepsilon}$, $\rho_\pm > 0$ and $\lambda \in S_{\pi-\varepsilon_0}$, then $\lambda_\pm \in S_{\pi-\varepsilon_0}$. Since $\rho_\pm > 0$, due to \refP{Prop arg}, we have 
$$\lambda_\pm + \rho_\pm \in S_{\pi - \varepsilon_0}.$$
From \refC{Cor DoreLabbas 1}, with $\mu = \lambda_+ + \rho_+$ and $\overline{L} = 2 L$, there exists $C_+ >0$ such that for all $z \in S_\varepsilon$ and all $\lambda \in S_{\pi-\varepsilon_0}$, we have
$$\left|\frac{q_{+}\left( 1-e^{-2L\sqrt{z+\lambda _{+}+\rho _{+}}}\right) }{\sqrt{z+\lambda _{+}+\rho _{+}}\left( 1+e^{-2L\sqrt{z+\lambda_{+}+\rho _{+}}}\right)}\right| \leqslant \frac{C_+}{\sqrt{|z|+|\lambda_+ + \rho_+|}} \leqslant \frac{C_+}{\sqrt{|\lambda_+ + \rho_+|}},$$
and due to \refP{Prop DoreLabbas}, we obtain
$$\left|\frac{q_{+}\left( 1-e^{-2L\sqrt{z+\lambda _{+}+\rho _{+}}}\right) }{\sqrt{z+\lambda _{+}+\rho _{+}}\left( 1+e^{-2L\sqrt{z+\lambda_{+}+\rho _{+}}}\right)}\right| \leqslant \frac{C_+}{\sqrt{|\lambda_+| + \rho_+} \sqrt{\cos\left(\frac{\pi - \varepsilon_0}{2}\right)}} \leqslant \frac{\sqrt{d_+}C_+}{\sqrt{|\lambda|} \sqrt{\sin\left(\frac{\varepsilon_0}{2}\right)}}.$$
Thus, there exists $R_+ >0$, depending only on $\varepsilon_0$, $d_+$ and $L$, such that for all $z \in S_\varepsilon$ and $\lambda\in S_{\pi-\varepsilon_0}$ with $|\lambda| \geqslant R_+$, we have
$$\left|\frac{q_{+}\left( 1-e^{-2L\sqrt{z+\lambda _{+}+\rho _{+}}}\right) }{\sqrt{z+\lambda _{+}+\rho _{+}}\left( 1+e^{-2L\sqrt{z+\lambda_{+}+\rho _{+}}}\right)}\right| \leqslant \frac{1}{2} - \frac{\sqrt{2}}{4};$$
similarly, there exists $R_- >0$, depending only on $\varepsilon_0$, $d_+$ and $L$, such that for all $z \in S_\varepsilon$ and $\lambda\in S_{\pi-\varepsilon_0}$ with $|\lambda| \geqslant R_-$, we obtain
$$\left|\frac{q_{-}\left( 1-e^{-2\ell \sqrt{z+\lambda_{-}+\rho _{-}}}\right) }{\sqrt{z+\lambda _{-}+\rho _{-}}\left( 1+e^{-2\ell \sqrt{z+\lambda _{-}+\rho _{-}}}\right) }\right| \leqslant \frac{1}{2} - \frac{\sqrt{2}}{4}.$$
Therefore, for all $z\in S_\varepsilon$ and $\lambda \in S_{\pi-\varepsilon_0}$ with $|\lambda| \geqslant R = \max(R_+,R_-)$, setting
$$z_1 = \frac{q_{+}\left( 1-e^{-2L\sqrt{z+\lambda _{+}+\rho _{+}}}\right) }{\sqrt{z+\lambda _{+}+\rho _{+}}\left( 1+e^{-2L\sqrt{z+\lambda_{+} + \rho_{+}}}\right) } \quad \text{and} \quad z_2 = \frac{q_{-}\left( 1-e^{-2\ell \sqrt{z+\lambda_{-}+\rho _{-}}}\right) }{\sqrt{z+\lambda _{-}+\rho _{-}}\left( 1+e^{-2\ell \sqrt{z+\lambda _{-}+\rho _{-}}}\right) },$$
we then obtain 
$$|f(z)| \geqslant \left|1 - |z_1+z_2|\right| \geqslant \left|1 - |z_1| - |z_2|\right| > 1 - \frac{1}{2} + \frac{\sqrt{2}}{4} - \frac{1}{2} + \frac{\sqrt{2}}{4} = \frac{\sqrt{2}}{2}.$$
\end{proof}

Now, in view to improve \refP{Prop f>0}, we consider that 
$$\lambda \in \overline{S_{\frac{\pi}{2}}} = \left\{ z \in \mathbb{C} \setminus \{0\} :|\arg(z)| \leqslant  \frac{\pi}{2} \right\} \cup \{0\},$$
which implies, from \refP{Prop arg} and \refP{Prop ineg arg}, that for all $z\in S_\varepsilon$, we have
$$z + \lambda_\pm + \rho_\pm \in S_{\frac{\pi}{2}},$$
hence
\begin{equation}\label{racine dans S_pi/4}
\sqrt{z + \lambda_\pm + \rho_\pm} \in S_{\frac{\pi}{4}}.
\end{equation}

\begin{proposition}\label{Prop f>0 pour Re(lambda)>0}
For all $z\in S_{\varepsilon}$ and $\lambda \in \overline{S_{\frac{\pi}{2}}}$, we have
$$\left|f(z)\right| > 1.$$
\end{proposition}

\begin{proof}
Set again
$$z_1 = \frac{q_{+}\left( 1-e^{-2L\sqrt{z+\lambda _{+}+\rho _{+}}}\right) }{\sqrt{z+\lambda _{+}+\rho _{+}}\left( 1+e^{-2L\sqrt{z+\lambda_{+} + \rho_{+}}}\right) } \quad \text{and} \quad z_2 = \frac{q_{-}\left( 1-e^{-2\ell \sqrt{z+\lambda_{-}+\rho _{-}}}\right) }{\sqrt{z+\lambda _{-}+\rho _{-}}\left( 1+e^{-2\ell \sqrt{z+\lambda _{-}+\rho _{-}}}\right) },$$
then due to \refP{Prop DoreLabbas 2} and \eqref{racine dans S_pi/4}, we obtain that
$$\begin{array}{lll}
|\arg(z_1)| &=& \dis \left| \arg \left( 1-e^{-2L\sqrt{z+\lambda _{+}+\rho _{+}}}\right)  - \arg \left( 1+e^{-2L\sqrt{z+\lambda_{+} + \rho_{+}}}\right) - \arg\left(\sqrt{z + \lambda_\pm + \rho_\pm}\right)\right| \medskip \\
& \leqslant & \dis \left| \arg \left( 1-e^{-2L\sqrt{z+\lambda _{+}+\rho _{+}}}\right)  - \arg \left( 1+e^{-2L\sqrt{z+\lambda_{+} + \rho_{+}}}\right)\right| + \left|\arg\left(\sqrt{z + \lambda_\pm + \rho_\pm}\right)\right| \medskip \\
& < & \dfrac{\pi}{4} + \dfrac{\pi}{4} = \dfrac{\pi}{2},
\end{array}$$
and similarly
$$|\arg(z_2)| < \frac{\pi}{2}.$$
Thus, Re$(z_1)>0$ and Re$(z_1)>0$. Therefore
$$|f(z)| = |1+z_1 + z_2| \geqslant \text{Re}(1+z_1 +z_2) > 1.$$
\end{proof}

\begin{remark} 
Let $z \in S_\varepsilon$ and $\lambda \in S_{\pi-\varepsilon}$. In view to give more precisions on the previous result we can show that
$$|f(z)| > \frac{\sqrt{2}}{2},$$
under the following assumption 
\begin{equation}\label{hyp Im(z+lambda) < pi}
|\text{Im}(z+\lambda_\pm)| \leqslant \frac{\pi^2\tan(\varepsilon)}{2(2 + \tan(\varepsilon))} \,\min\left(\frac{1}{L^2},\frac{1}{\ell^2}\right) .
\end{equation}
Indeed, the algebraic formula for the square roots of a complex number gives us 
$$\left|\text{Im}\left(2L\sqrt{z+\lambda_+ + \rho_+}\right)\right| = 2L\sqrt{\frac{\left|z+\lambda_+ + \rho_+ \right| - \text{Re}\left(z+\lambda_+ + \rho_+\right)}{2}},$$
and
$$\left|\text{Im}\left(2\ell\sqrt{z+\lambda_- + \rho_-}\right)\right| = 2\ell\sqrt{\frac{\left|z+\lambda_- + \rho_- \right| - \text{Re}\left(z+\lambda_- + \rho_-\right)}{2}}.$$
When Re$(z+\lambda_\pm+\rho_\pm) \geqslant 0$, since
$$|z+\lambda_\pm + \rho_\pm| \leqslant |\text{Re}(z+\lambda_\pm + \rho_\pm)| + |\text{Im}(z+\lambda_\pm + \rho_\pm)|,$$
we obtain
$$\sqrt{\frac{|z+\lambda_\pm + \rho_\pm| - \text{Re}(z+\lambda_\pm + \rho_\pm)}{2}} \leqslant \sqrt{\frac{|\text{Im}(z+\lambda_\pm + \rho_\pm)|}{2}}.$$
When Re$(z+\lambda_\pm+\rho_\pm) < 0$, we have
$$\tan(\pi-(\pi-\arg(z+\lambda_\pm + \rho_\pm))) = \frac{|\text{Im}(z+\lambda_\pm + \rho_\pm)|}{|\text{Re}(z+\lambda_\pm + \rho_\pm)|},$$
and since $z+\lambda_\pm + \rho_\pm \in S_{\pi - \varepsilon}$, it follows 
$$|\text{Re}(z+\lambda_\pm + \rho_\pm)| \leqslant \frac{|\text{Im}(z+\lambda_\pm + \rho_\pm)|}{\tan(\varepsilon)},$$
hence
$$\begin{array}{lll}
\dis |z+\lambda_\pm + \rho_\pm| - \text{Re}(z+\lambda_\pm + \rho_\pm) & \leqslant & \dis 2 |\text{Re}(z+\lambda_\pm + \rho_\pm)| + |\text{Im}(z+\lambda_\pm + \rho_\pm)| \\ \ecart
& \leqslant & \dis |\text{Im}(z+\lambda_\pm + \rho_\pm)| \left( 1 + \frac{2}{\tan(\varepsilon)}\right).
\end{array}$$
Then
$$\sqrt{\frac{|z+\lambda_\pm + \rho_\pm| - \text{Re}(z+\lambda_\pm + \rho_\pm)}{2}} \leqslant \sqrt{\frac{|\text{Im}(z+\lambda_\pm + \rho_\pm)|}{2}} \sqrt{1 + \frac{2}{\tan(\varepsilon)}}.$$
Thus, from \eqref{hyp Im(z+lambda) < pi}, we deduce that 
$$\left|\text{Im}\left(2L\sqrt{z+\lambda_+ + \rho_+}\right)\right| \leqslant \pi \quad \text{and} \quad \left|\text{Im}\left(2\ell\sqrt{z+\lambda_- + \rho_-}\right)\right| \leqslant \pi.$$
Let $\lambda \neq 0$. We have to consider the two following cases: 
\begin{enumerate}
\item $-\varepsilon \leqslant \arg(\lambda_\pm) < \pi - \varepsilon$,
\item $-\pi + \varepsilon < \arg(\lambda_\pm) \leqslant \varepsilon$.
\end{enumerate}
Let $-\varepsilon \leqslant \arg(\lambda_\pm) < \pi - \varepsilon$.
Since $|\arg(z)| < \varepsilon$, from Proposition \ref{Prop ineg arg}, we have
$$-\varepsilon \leqslant \min\left(\arg(\lambda_\pm), \arg(z)\right) \leqslant \arg(z + \lambda_\pm) \leqslant \max\left(\arg(\lambda_\pm), \arg(z)\right) < \pi - \varepsilon,$$
and thus, if $\arg(z + \lambda_\pm) \neq 0$, from Proposition \ref{Prop arg}, we deduce that
$$\left|\arg(z + \lambda_\pm + \rho_\pm )\right| < \left|\arg(z + \lambda_\pm)\right| < \pi - \varepsilon,$$
and if $\arg(z+\lambda_\pm) = 0$, then $\arg(z+\lambda_\pm + \rho_\pm) = 0$. 
Moreover, when $\arg(z + \lambda_\pm + \rho_\pm ) < 0$, then due to Proposition \ref{Prop arg}, we have
$$-\varepsilon \leqslant \arg(z + \lambda_\pm ) < \arg(z + \lambda_\pm + \rho_\pm ).$$
Therefore, we obtain that
\begin{equation}\label{z+lambda+rho < pi- eps0}
-\varepsilon \leqslant \arg(z + \lambda_\pm + \rho_\pm ) < \pi- \varepsilon.
\end{equation}
Hence, setting $\overline{L} = 2L$ or $2\ell$, we deduce that 
\begin{equation}\label{theta+}
-\frac{\varepsilon}{2} \leqslant \arg\left(\sqrt{z + \lambda_\pm + \rho_\pm}\right) = \arg\left(\overline{L}\sqrt{z + \lambda_\pm + \rho_\pm}\right) < \frac{\pi}{2} - \frac{\varepsilon}{2}.
\end{equation}
We set
$$z_1 = \frac{q_{+}\left( 1-e^{-2L\sqrt{z+\lambda _{+}+\rho _{+}}}\right) }{\sqrt{z+\lambda _{+}+\rho _{+}}\left( 1+e^{-2L\sqrt{z+\lambda_{+} + \rho_{+}}}\right) } \quad \text{and} \quad z_2 = \frac{q_{-}\left( 1-e^{-2\ell \sqrt{z+\lambda_{-}+\rho _{-}}}\right) }{\sqrt{z+\lambda _{-}+\rho _{-}}\left( 1+e^{-2\ell \sqrt{z+\lambda _{-}+\rho _{-}}}\right) }.$$
Then, due to Proposition \ref{Prop DoreLabbas}, it follows that
$$\left|f(z)\right| = |1 + z_1 + z_2| \geqslant \left(1 + |z_1 + z_2|\right) \left|\cos\left(\frac{\arg(z_1 + z_2)}{2}\right)\right| > \left|\cos\left(\frac{\arg(z_1 + z_2)}{2}\right)\right|.$$
Moreover, due to \eqref{hyp Im(z+lambda) < pi} and \eqref{theta+}, we can use Corollary \ref{Cor DoreLabbas 2}, with $\alpha = \dfrac{\pi}{2}$ and $\beta = \dfrac{\varepsilon}{2}$; it follows that
 
$$- \frac{\varepsilon}{2} \leqslant \arg\left(1 - e^{-2L\sqrt{z + \lambda_+ + \rho_+}}\right) - \arg\left(1 + e^{-2L\sqrt{z + \lambda_+ + \rho_+}}\right) < \frac{\pi}{2} - \frac{\varepsilon}{2},$$
with
$$- \frac{\varepsilon}{2} \leqslant \arg\left(1 - e^{-2\ell\sqrt{z + \lambda_- + \rho_-}}\right) - \arg\left(1 + e^{-2\ell\sqrt{z + \lambda_- + \rho_-}}\right) < \frac{\pi}{2} - \frac{\varepsilon}{2}.$$
Since 
\begin{equation*}
\arg(z_1) = \arg\left(1 - e^{-2L\sqrt{z + \lambda_+ + \rho_+}}\right) - \arg\left(1 + e^{-2L\sqrt{z + \lambda_+ + \rho_+}}\right) - \arg\left(\sqrt{z + \lambda_+ + \rho_+} \right),
\end{equation*}
and
\begin{equation*}
\arg(z_2) = \arg\left(1 - e^{-2\ell \sqrt{z + \lambda_- + \rho_-}}\right) - \arg\left(1 + e^{-2\ell \sqrt{z + \lambda_- + \rho_-}}\right) - \arg\left(\sqrt{z + \lambda_- + \rho_-}\right),
\end{equation*}
it follows that 
$$- \frac{\varepsilon}{2} - \arg\left(\sqrt{z + \lambda_+ + \rho_+}\right) \leqslant \arg(z_1) < \frac{\pi}{2} - \frac{\varepsilon}{2} - \arg\left(\sqrt{z + \lambda_+ + \rho_+} \right),$$
and
$$- \frac{\varepsilon}{2} - \arg\left(\sqrt{z + \lambda_- + \rho_-}\right) \leqslant \arg(z_2) < \frac{\pi}{2} - \frac{\varepsilon}{2} - \arg\left(\sqrt{z + \lambda_- + \rho_-} \right).$$
Then, due to \eqref{theta+}, we obtain 
$$-\frac{\pi}{2} < - \frac{\varepsilon}{2} - \arg\left(\sqrt{z + \lambda_+ + \rho_+}\right) \leqslant \arg(z_1) < \frac{\pi}{2} - \frac{\varepsilon_0}{2} - \arg\left(\sqrt{z + \lambda_+ + \rho_+} \right) \leqslant \frac{\pi}{2},$$
with
$$- \frac{\pi}{2} < - \frac{\varepsilon}{2} - \arg\left(\sqrt{z + \lambda_- + \rho_-}\right) \leqslant \arg(z_2) < \frac{\pi}{2} - \frac{\varepsilon}{2} - \arg\left(\sqrt{z + \lambda_- + \rho_-} \right) \leqslant \frac{\pi}{2}.$$
Thus, $\arg(z_1),\arg(z_2) \in \left(-\dfrac{\pi}{2}, \dfrac{\pi}{2}\right)$. Then, in virtue of Proposition \ref{Prop ineg arg}, we deduce that
$$
- \frac{\pi}{2} < \min\left(\arg(z_1),\arg(z_2)\right) \leqslant \arg(z_1+z_2) \leqslant \max\left(\arg(z_1),\arg(z_2)\right) < \frac{\pi}{2}.
$$
We then obtain 
$$\left| \arg(z_1+z_2) \right| < \frac{\pi}{2},$$
and
$$\cos\left(\frac{\arg(z_1 + z_2)}{2}\right) > \cos\left(\frac{\pi}{4}\right) = \frac{\sqrt{2}}{2} > 0.$$
Therefore, for all $z \in S_{\varepsilon}$ such that \eqref{hyp Im(z+lambda) < pi} holds, we have 
$$|f(z)| > \frac{\sqrt{2}}{2}.$$
Now, when $-\pi + \varepsilon < \arg(\lambda_\pm) \leqslant \varepsilon$, then
$$-\frac{\pi}{2} + \frac{\varepsilon}{2} < \arg(\sqrt{z + \lambda_\pm + \rho_\pm}) = \arg(\overline{L}\sqrt{z + \lambda_\pm + \rho_\pm}) \leqslant \frac{\varepsilon}{2},$$
and due to Corollary \ref{Cor DoreLabbas 2}, we have $\arg(z_1),\arg(z_2) \in \left(-\dfrac{\pi}{2}, \dfrac{\pi}{2}\right)$. Thus, from Proposition \ref{Prop ineg arg}, we deduce the expected result. Finally, when $\lambda = 0$, it is clear that, following the same steps, we obtain a similar result.
\end{remark}

\begin{proposition}\label{Prop inv det}
Let $\lambda \in \overline{S_{\frac{\pi}{2}}}$. Then, operator $D$ is boundedly invertible with
\begin{equation*}
D^{-1}=D_{\ast }^{-1}\left( I+e^{2LB_{+}}\right) ^{-1}\left( I+e^{2\ell
B_{-}}\right) ^{-1},
\end{equation*}
and there exists $C>0$, independent of $\lambda$, such that
$$\left\|D^{-1}\right\|_{\mathcal{L}(E)} \leqslant C.$$
\end{proposition}
\begin{proof}
From \refP{Prop f>0 pour Re(lambda)>0}, $f$ does not vanish on $S_\varepsilon$ and $1/f$ is bounded. Moreover
\begin{equation*}
\left\Vert \dfrac{1}{f}\right\Vert _{\infty }\leqslant 1,
\end{equation*}
and taking into account \refC{Cor DoreLabbas 1}, we have
$$|f(z) - 1| = O\left(\frac{1}{|z|^{1/2}}\right), \quad z \in S_{\varepsilon} \quad \text{and} \quad |z| \longrightarrow + \infty.$$
Then, from \eqref{MoinsAbip}, there exists $C>0$ independent of $\lambda$ such that
$$\left\|\left(\frac{1}{f}\right)(-A) \right\|_{\mathcal{L}(E)} \leqslant C \left\Vert \dfrac{1}{f}\right\Vert _{\infty } \leqslant C.$$
Finally, from  \refP{Prop Dore-Labbas calcul fonctionnel}, we deduce that
\begin{equation*}
D_{\ast }^{-1}=\left( \frac{1}{f}\right) (-A) \in \mathcal{L}(E) \quad \text{and} \quad \left\Vert D_{\ast }^{-1}\right\Vert _{\mathcal{L}(E)}\leqslant C,
\end{equation*}
and using \refL{Inv (I-exp2Q)}, we obtain the expected result.
\end{proof}
\begin{remark}\label{Rem interpolation sapces}
Now, from equality $D_{\ast }A^{-1}=A^{-1}D_{\ast }$, it follows that
\begin{equation*}
D_{\ast }^{-1}A=AD_{\ast }^{-1},
\end{equation*}
on $D(A)$, hence $D_{\ast }^{-1}$ is a bounded operator from $D(A)$ into
itself. Therefore, by interpolation $D_{\ast }^{-1}$ is bounded from any
interpolation space $(D(A),E)_{\alpha ,\beta }$, for all $\alpha \in (0,1)$ and $\beta \in [1,+\infty]$ (see the definition in \cite{Grisvard}) into itself and clearly we have also the same estimate
\begin{equation*}
\left\Vert D_{\ast }^{-1}\right\Vert _{\mathcal{L}((D(A),E)_{\alpha ,\beta
})}\leqslant C.
\end{equation*}
\end{remark}

\subsubsection{Resolution of the system}

Assume that $\lambda \in \overline{S_{\frac{\pi}{2}}}$. Recall that
\begin{equation*}
\left\{ 
\begin{array}{l}
\left[ \left( I+e^{2\ell B_{-}}\right) -q_{-}B_{-}^{-1}\left( I-e^{2\ell
B_{-}}\right) \right] k_{-}+q_{-}B_{-}^{-1}\left( I-e^{2LB_{+}}\right)
j_{+}=B_{-}^{-1}(\Pi'), \\ 
\\ 
q_{+}B_{+}^{-1}\left( I-e^{2\ell B_{-}}\right) k_{-}+\left[ \left(
I+e^{2LB_{+}}\right) -q_{+}B_{+}^{-1}\left( I-e^{2LB_{+}}\right) \right]
j_{+}=B_{+}^{-1}(\Pi''),
\end{array}
\right.
\end{equation*}
where
\begin{equation*}
\left\{ 
\begin{array}{lll}
(\Pi') & = & \displaystyle v_{-}'(g_{-})(0)-B_{-}e^{\ell
B_{-}}v_{-}(g_{-})(-\ell)+q_{-}e^{LB_{+}}v_{+}(g_{+})(L)-q_{-}v_{+}(g_{+})(0)\medskip \\ 
&&\displaystyle -q_{-}e^{\ell B_{-}}v_{-}(g_{-})(-\ell)+q_{-}v_{-}(g_{-})(0), \\ 
\\ 
(\Pi'') & = & \displaystyle -v_{+}'(g_{+})(0)-B_{+}e^{LB_{+}}v_{+}(g_{+})(L)-q_{+}e^{LB_{+}}v_{+}(g_{+})(L)+q_{+}v_{+}(g_{+})(0),\medskip
\\ 
&& \displaystyle + q_{+}e^{\ell B_{-}}v_{-}(g_{-})(-\ell)-q_{+}v_{-}(g_{-})(0),
\end{array}
\right.
\end{equation*}
therefore
\begin{eqnarray*}
k_{-} &=& D^{-1}\left\vert 
\begin{array}{rr}
B_{-}^{-1}(\Pi') & q_{-}B_{-}^{-1}(I-e^{2LB_{+}})\medskip \\ 
B_{+}^{-1}(\Pi'') & \left[ \left( I+e^{2LB_{+}}\right)-q_{+}B_{+}^{-1}\left( I-e^{2LB_{+}}\right) \right]
\end{array}
\right\vert \\
&& \\
&=&D^{-1}\left[ B_{-}^{-1}\left[ \left( I+e^{2LB_{+}}\right)
-q_{+}B_{+}^{-1}\left( I-e^{2LB_{+}}\right) \right] (\Pi')-q_{-}B_{+}^{-1}B_{-}^{-1}(I-e^{2LB_{+}})(\Pi'')\right] ,
\end{eqnarray*}
and
\begin{eqnarray*}
j_{+} &=&D^{-1}\left\vert 
\begin{array}{ll}
\left[ \left( I+e^{2\ell B_{-}}\right) -q_{-}B_{-}^{-1}\left( I-e^{2\ell
B_{-}}\right) \right] & B_{-}^{-1}(\Pi')\medskip \\ 
q_{+}B_{+}^{-1}\left( I-e^{2\ell B_{-}}\right) & B_{+}^{-1}(\Pi'')
\end{array}
\right\vert \\
&& \\
&=&D^{-1}\left[ \left[ \left( I+e^{2\ell B_{-}}\right)
-q_{-}B_{-}^{-1}\left( I-e^{2\ell B_{-}}\right) \right] B_{+}^{-1}(\Pi'')-q_{+}B_{+}^{-1}\left( I-e^{2\ell B_{-}}\right)
B_{-}^{-1}(\Pi')\right] .
\end{eqnarray*}
We then deduce
\begin{equation*}
\left\{ \begin{array}{lll}
j_{-} & = & \displaystyle -e^{\ell B_{-}}D^{-1}B_{-}^{-1}\left[ \left( I+e^{2LB_{+}}\right) -q_{+}B_{+}^{-1}\left( I-e^{2LB_{+}}\right) \right] \medskip (\Pi') \\ 
&& \displaystyle + q_{-}e^{\ell B_{-}}D^{-1}B_{+}^{-1}B_{-}^{-1}(I-e^{2LB_{+}})(\Pi'')-v_{-}(g_{-})(-\ell ), \\ \\
 
k_{+} & = & \displaystyle -e^{LB_{+}}D^{-1}\left[ \left( I+e^{2\ell B_{-}}\right)
-q_{-}B_{-}^{-1}\left( I-e^{2\ell B_{-}}\right) \right] B_{+}^{-1}(\Pi'') \\ 
&& \displaystyle + q_{+}e^{LB_{+}}D^{-1}B_{+}^{-1}\left( I-e^{2\ell B_{-}}\right)
B_{-}^{-1}(\Pi')-v_{+}(g_{+})(L).
\end{array}\right.
\end{equation*}
Finally, we obtain for $a.e.$ $x\in (-\ell ,0) $
\begin{eqnarray*}
w_{-}\left( x\right)&=&D^{-1}\left( e^{-xB_{-}}-e^{(x+2\ell )B_{-}}\right) \left[ \left(
I+e^{2LB_{+}}\right) -q_{+}B_{+}^{-1}\left( I-e^{2LB_{+}}\right) \right]
B_{-}^{-1}(\Pi')\medskip \\
&&+q_{-}D^{-1}\left( e^{(x+2\ell )B_{-}}-e^{-xB_{-}}\right)
B_{-}^{-1}(I-e^{2LB_{+}})B_{+}^{-1}(\Pi'')\medskip \\
&&-e^{(x+\ell )B_{-}}v_{-}(g_{-})(-\ell ) + v_{-}(g_{-})(x),
\end{eqnarray*}
and for $a.e.$ $x\in (0,L)$
\begin{eqnarray*}
w_{+}\left( x\right) &=&D^{-1}\left( e^{xB_{+}}-e^{(2L-x)B_{+}}\right) \left[ \left( I+e^{2\ell B_{-}}\right) -q_{-}B_{-}^{-1}\left( I-e^{2\ell B_{-}}\right) \right]B_{+}^{-1}(\Pi'')\medskip \\
&&+q_{+}D^{-1}\left( e^{(2L-x)B_{+}}-e^{xB_{+}}\right) B_{+}^{-1}\left(
I-e^{2\ell B_{-}}\right) B_{-}^{-1}(\Pi')\medskip \\
&&-e^{(L-x)B_{+}}v_{+}(g_{+})(L) + v_{+}(g_{+})(x).
\end{eqnarray*}

\subsubsection{Optimal regularity of w$_{-}$ and w$_{+}$}\label{Section optimal reg}

Let $\lambda \in \overline{S_{\frac{\pi}{2}}}$. Since $B_{+}$ generates an analytic semigroup in $E$, we recall the following known results
\begin{equation}\label{RESULT1}
\left\{ 
\begin{array}{l}
x\longmapsto e^{xB_{+}}\psi \in L^{p}\left( 0,L;E\right) \text{ for all }\psi
\in E, \medskip \\ 
x\longmapsto B_{+}^{n}e^{xB_{+}}\psi \in L^{p}\left( 0,L;E\right)
\Longleftrightarrow \psi \in \left( D\left( B_{+}^{n}\right) ,E\right) _{\frac{1}{np},p},
\end{array}
\right.  
\end{equation}
where $p\in \left( 1,+\infty \right) $ and $n\in \mathbb{N}\backslash
\left\{ 0\right\} $; see the Theorem in \cite{Triebel}, p. 96.

We have the same result for $B_{-}$ on $(-\ell ,0)$. Note that, for these two results, we do not need assumption \eqref{UMD}. 

Let us recall the following well-known important result proved in \cite{Dore-Venni}.

\begin{theorem}\label{Th DorVen}
Let $X$ be a UMD Banach space, $-Q\in$ BIP\,$(X,\theta)$ with $\theta\in(0,\pi/2)$ and \mbox{$g\in L^p(a,b;X)$}. Then, for almost every $x\in(a,b)$, we have
$$
\int_a^x e^{(x-s)Q}g(s)\,ds \in D(Q)\quad\text{and}\quad \int_x^b e^{(s-x)Q}g(s)\,ds \in D(Q).
$$
Moreover,
$$
x\longmapsto Q\int_a^x e^{(x-s)Q}g(s)\,ds \in L^p(a,b;X)\quad\text{and}\quad x\longmapsto Q\int_x^b e^{(s-x)Q}g(s)\,ds \in L^p(a,b;X).
$$
\end{theorem}
We are applying these results to our operators $B_\pm$. For all $\lambda \in \overline{S_{\frac{\pi}{2}}}$, we have
$$|\arg(\lambda_\pm + \rho_\pm)| < \frac{\pi}{2}.$$ 
Now, applying Theorem 2.4, p. 408 in \cite{monniaux}, on the sum
$$-A + (\lambda_\pm + \rho_\pm) I,$$
we obtain that
$$- A_\pm\in \text{BIP}\left(E,\frac{\pi}{2}\right),$$
since 
$$\varepsilon + |\arg(\lambda_\pm + \rho_\pm)| < \varepsilon + \frac{\pi}{2} < \pi.$$ 
We deduce that 
$$- B_\pm\in \text{BIP}\left(E,\frac{\pi}{4}\right),$$ 
from Proposition $3.2.1,$ e), p.~71 in \cite{Haase}. We then obtain the following lemma by taking $Q=B_\pm$.

\begin{lemma}\label{Dore Venni}
Let $h_-\in L^{p}\left(-\ell ,0;E\right)$ and $h_+\in L^{p}\left( 0,L;E\right)$ with $1<p<+\infty $. Assume that \eqref{UMD}, \eqref{0 dans rho de A} and \eqref{MoinsAbip} hold. Then, we have
$$\left\{\begin{array}{l}
x\longmapsto B_{-}\displaystyle \int_{-\ell }^{x}e^{\left( x-s\right) B_{-}}h_-(s) ds\in L^{p}\left(-\ell ,0;E\right), \medskip \\
x\longmapsto B_{-}\displaystyle \int_{x}^{0}e^{\left( s-x\right) B_{-}}h_-\left( s\right) ds \in L^{p}\left( -\ell ,0;E\right),
\end{array}\right.$$
and
$$\left\{\begin{array}{l}
x\longmapsto B_{+}\displaystyle \int_{0}^{x}e^{\left( x-s\right) B_{+}}h_+\left( s\right) ds\in L^{p}\left( 0,L;E\right), \medskip \\
x\longmapsto B_{+}\displaystyle \int_{x}^{L}e^{\left( s-x\right) B_{+}}h_+\left( s\right) ds\in L^{p}\left( 0,L;E\right).
\end{array}\right.$$
\end{lemma}

\begin{lemma}\label{Dore Venni 2}
Let $h_-\in L^{p}\left( -\ell ,0;E\right)$ and $h_+ \in L^{p}\left( 0,L;E\right)$
with $1<p<+\infty$. Assume that \eqref{UMD}, \eqref{0 dans rho de A} and \eqref{MoinsAbip} hold. Then, we have

\begin{enumerate}
\item $\displaystyle \int_{0}^{L}e^{sB_{+}}h_+\left( s\right) ds$ and $\displaystyle \int_{0}^{L}e^{(L-s)B_{+}}h_+\left( s\right) ds$ belong to $\left(D(B_{+}),E\right) _{\frac{1}{p},p}=\left( D(B),E\right) _{\frac{1}{p},p}, \medskip $

\item $\displaystyle \int_{-\ell }^{0}e^{(s+\ell)B_{-}}h_-\left( s\right) ds$ and $\displaystyle \int_{-\ell}^{0}e^{-sB_{-}}h_-\left( s\right) ds$ belong to $\left( D(B_{-}),E\right) _{\frac{1}{p},p}=\left( D(B),E\right) _{\frac{1}{p},p}.$
\end{enumerate}
\end{lemma}
\begin{proof}
Let us indicate the proof of the first statement for instance. Consider the
function 
\begin{equation*}
\psi _{1}(x)=\displaystyle \int_{0}^{x}e^{(x-s)B_{+}}h_+\left( s\right) ds;
\end{equation*}
then, from Theorem \ref{Th DorVen}, we know that
\begin{equation*}
\psi _{1}\in W^{1,p}\left( 0,L;E\right) \cap L^{p}\left( 0,L;D(B_{+})\right)
;
\end{equation*}
by using the notation in \cite{Grisvard}, pp. 677-678 for the spaces of
traces, we deduce that 
\begin{equation*}
\psi _{1}(L)\in T_{0}^{1}\left( p,0,D(B_{+}),E\right) =\left(
D(B_{+}),E\right) _{\frac{1}{p},p}=\left( D(B),E\right) _{\frac{1}{p},p},
\end{equation*}
here, the Poulsen condition is verified since $0<1/p<1$. By considering the function
\begin{equation*}
\psi _{2}(x)=\displaystyle \int_{x}^{L}e^{(s-x)B_{+}}h_+\left( s\right) ds,
\end{equation*}
we get 
\begin{equation*}
\psi _{2}(0)\in T_{0}^{1}\left( p,0,D(B_{+}),E\right) =\left(
D(B_{+}),E\right) _{\frac{1}{p},p}=\left( D(B),E\right) _{\frac{1}{p},p}.
\end{equation*}
Statement 2 is obtained analogously.
\end{proof}

\begin{proposition}\label{Prop Existence-Unicité w}
Let $f \in L^{p}\left(-\ell,L;E\right)$ with $1<p<+\infty$. Assume that \eqref{UMD}, \eqref{0 dans rho de A} and \eqref{MoinsAbip} hold. Then, for all $\lambda \in \overline{S_{\frac{\pi}{2}}}$, there exists a unique solution $w \in D(\mathcal{S})$ of equation \eqref{SpectralEquation}.
\end{proposition}

\begin{proof}
Now we must show that
\begin{equation*}
\left\{ \begin{array}{l}
w_{-}\in W^{2,p}\left( -\ell ,0;E\right) \cap L^{p}\left( -\ell,0;D(A)\right), \medskip \\ 
w_{+}\in W^{2,p}\left( 0,L;E\right) \cap L^{p}\left( 0,L;D(A)\right).
\end{array}\right.
\end{equation*}
It is not difficult to see that all the boundary and transmission conditions in \eqref{PA} are verified by $w_-$ and $w_+$. 

From Proposition 4.4 in \cite{DoreLabbas}, p. 1878, to prove that
$$x \longmapsto w_-(x) = e^{(x+l)B_-} j_- + e^{-x b_-} k_- + v_- (g_-)(x)\in W^{2,p}\left( -\ell ,0;E\right) \cap L^{p}\left( -\ell,0;D(A)\right),$$
it suffices to show that $j_-$ and $k_-$ belong to $(D(B^2_-),E)_{\frac{1}{2p},p}$. Recall that due to \eqref{syst Prop reg}, we have
$$j_- = - e^{\ell B_{-}}k_{-}-v_{-}(g_{-})(-\ell ).$$
It is clear that 
$$e^{\ell B_{-}}k_{-} \in D(B^2_-) \subset (D(B^2_-),E)_{\frac{1}{2p},p}.$$
Moreover, due to \eqref{v_pm(g_pm) (x)}, we have
$$v_{-}(g_{-})(-\ell) = \frac{1}{2}B_{-}^{-1} \int_{-\ell}^{0}e^{(t+\ell)B_{-}}g_{-}(t)dt,$$
thus, from \refL{Dore Venni 2}, it follows that
$$B_{-}v_{-}(g_{-})(-\ell) \in (D(B_-),E)_{\frac{1}{p},p},$$
hence
$$v_{-}(g_{-})(-\ell) \in (D(B_-),E)_{1+\frac{1}{p},p} = (D(B^2_-),E)_{\frac{1}{2p},p}.$$
Furthermore, recall that 
$$k_- = D^{-1}\left[ B_{-}^{-1}\left[ \left( I+e^{2LB_{+}}\right)
-q_{+}B_{+}^{-1}\left( I-e^{2LB_{+}}\right) \right] (\Pi')-q_{-}B_{+}^{-1}B_{-}^{-1}(I-e^{2LB_{+}})(\Pi'')\right],$$
where
\begin{eqnarray*}
(\Pi') &=&-B_{-}e^{\ell B_{-}}v_{-}(g_{-})(-\ell ) + v'_{-}(g_{-})(0) + q_{-}e^{LB_{+}}v_{+}(g_{+})(L) - q_{-}v_{+}(g_{+})(0)\medskip \\
&& - q_{-}e^{\ell B_{-}}v_{-}(g_{-})(-\ell ) + q_{-}v_{-}(g_{-})(0),
\end{eqnarray*}
and 
\begin{eqnarray*}
(\Pi'') &=&-B_{+}e^{LB_{+}}v_{+}(g_{+})(L) - v'_{+}(g_{+})(0) - q_{+}e^{LB_{+}}v_{+}(g_{+})(L) + q_{+}v_{+}(g_{+})(0)\medskip \\
&& + q_{+}e^{\ell B_{-}}v_{-}(g_{-})(-\ell ) - q_{+}v_{-}(g_{-})(0).
\end{eqnarray*}
From \refR{Rem interpolation sapces}, interpolation spaces are invariant for $D^{-1}$, therefore, in order to prove that $k_- \in (D(B^2_-),E)_{\frac{1}{2p},p}$, it is sufficient to show that 
$$(\Pi'),~(\Pi'') \in (D(B_-),E)_{\frac{1}{p},p} = (D(B),E)_{\frac{1}{p},p} = (D(B_+),E)_{\frac{1}{p},p}.$$
For $(\Pi')$, we have
\begin{equation*}
-B_{-}e^{\ell B_{-}}v_{-}(g_{-})(-\ell)-q_{-}e^{\ell B_{-}}v_{-}(g_{-})(-\ell ) \in D(B^2_-) \subset (D(B_-),E)_{\frac{1}{p},p},
\end{equation*}
and
$$q_{-}e^{LB_{+}}v_{+}(g_{+})(L) \in D(B^2_+) \subset (D(B_+),E)_{\frac{1}{p},p}.$$
Similarly, for $(\Pi'')$, we have
\begin{equation*}
-B_{+}e^{LB_{+}}v_{+}(g_{+})(L)-q_{+}e^{LB_{+}}v_{+}(g_{+})(L) \in D(B^2_+) \subset (D(B_+),E)_{\frac{1}{p},p},
\end{equation*}
and
$$q_{+}e^{\ell B_{-}}v_{-}(g_{-})(-\ell ) \in D(B^2_-) \subset (D(B_-),E)_{\frac{1}{p},p}.$$
Then, it remains to prove that
$$v'_{-}(g_{-})(0) - q_{-}v_{+}(g_{+})(0) + q_{-}v_{-}(g_{-})(0) \in (D(B),E)_{\frac{1}{p},p},$$
and
$$- v'_{+}(g_{+})(0) + q_{+}v_{+}(g_{+})(0) - q_{+}v_{-}(g_{-})(0) \in (D(B),E)_{\frac{1}{p},p}.$$
From \refL{Dore Venni 2}, we have 
\begin{equation*}
v'_{-}(g_{-})(0) = B_{-}v_{-}(g_{-})(0)=\dfrac{1}{2}\displaystyle \int_{-\ell }^{0}e^{-tB_{-}}g_{-}(t) dt\in \left(D(B_{-}),E\right) _{\frac{1}{p},p},
\end{equation*}
and
\begin{equation*}
-v'_{+}(g_{+})(0) = B_{+}v_{+}(g_{+})(0)=\dfrac{1}{2}\displaystyle \int_{0}^{L}e^{tB_{+}}g_{+}(t) dt \in \left( D(B_{+}),E\right) _{\frac{1}{p},p},
\end{equation*}
hence
\begin{equation*}
v_{-}(g_{-})(0)=\frac{1}{2}B_{-}^{-1}\displaystyle \int_{-\ell }^{0}e^{-tB_{-}}g_{-}(t)dt \in \left( D(B_{-}),E\right) _{1+\frac{1}{p},p} \subset \left( D(B_{-}),E\right) _{\frac{1}{p},p},
\end{equation*}
and 
\begin{equation*}
v_{+}(g_{+})(0)=\frac{1}{2}B_{+}^{-1}\displaystyle \int_{0}^{L}e^{tB_{+}}g_{+}\left(t\right) dt \in \left( D(B_{+}),E\right) _{1+\frac{1}{p},p} \subset \left( D(B_{+}),E\right) _{\frac{1}{p},p}.
\end{equation*}
The proof for $w_+$ is analogous. Therefore $w \in D(\mathcal{S})$.
\end{proof}

\subsection{Estimate of the resolvent operator}

In all the sequel, $\lambda \in \overline{S_{\frac{\pi}{2}}}$.

\subsubsection{Some sharp estimates}

Recall that
$$B_- = -\sqrt{-A + \rho_- I + \lambda_- I} \quad \text{and} \quad B_+ = -\sqrt{-A + \rho_+ I + \lambda_+ I},$$
where
$$\lambda_\pm = \frac{\lambda}{d_\pm} \quad \text{and} \quad  \rho_\pm = \frac{r_\pm}{d_\pm}.$$
Recall that Lemma 2.6, p. 103 in \cite{Dore-Yakubov} gives the following result
\begin{equation}\label{YakubovEstimate}
\left\{ 
\begin{array}{l}
\exists \, C > 0,\exists\, c > 0,~\forall\, \kappa \in \mathbb{R},~\forall\, t \geqslant t_0 > 0,\forall
\, \lambda \in S_{\pi -\varepsilon _{0}}\cup\{0\}: \medskip \\ 
\left\Vert \left(-B_\pm\right)^{\kappa} e^{t B_\pm}\right\Vert_{\mathcal{L}(E)} \leqslant
C e^{- c t \left\vert \lambda_\pm + \rho_\pm \right\vert^{1/2}}.
\end{array}\right. 
\end{equation} 

\begin{proposition} \label{Prop Estimate U et V}
Let $g\in L^{p}\left(0,L;E\right)$ and $h\in L^{p}\left( -\ell ,0;E\right)$, $1<p<+\infty$. Set 
\begin{equation*}
U(g)(x)=\displaystyle \int_{0}^{L}e^{\left\vert x-t\right\vert B_{+}}g(t)dt \quad \text{and} \quad V(h)(x)=\displaystyle \int_{-\ell }^{0}e^{\left\vert x-t\right\vert B_{-}}h(t)dt.
\end{equation*}
Then, there exists a constant $C >0$ independent of $\lambda$ such that
\begin{equation*}
\left\{\begin{array}{lll}
\displaystyle\left\Vert U(g)\right\Vert _{L^{p}\left( 0,L;E\right) } & \leqslant & \displaystyle \frac{C \sqrt{d_{+}}}{\sqrt{\left\vert \lambda +r_{+}\right\vert +d_{+}}}\left\Vert g\right\Vert _{L^{p}\left( 0,L;E\right) }, \medskip \\

\displaystyle \left\Vert V(h)\right\Vert _{L^{p}\left( -\ell ,0;E\right) } & \leqslant & \displaystyle \frac{C \sqrt{d_{-}}}{\sqrt{\left\vert \lambda +r_{-}\right\vert +d_{-}}} \left\Vert h\right\Vert _{L^{p}\left( -\ell ,0;E\right) }.
\end{array}\right. 
\end{equation*}
\end{proposition}
\begin{proof}
We will prove the estimate, for instance, for $B_+$. Let
$$g\in \mathcal{D}\left(0,L ;\mathcal{L}(E)\right)\subset \mathscr{S}\left( \mathbb{R};\mathcal{L}(E)\right),$$
where $\mathcal{D}\left(0,L;\mathcal{L}(E)\right)$ is the space of all vector-valued test functions on $\mathcal{L}(E)$ and $\mathscr{S}\left( \mathbb{R};\mathcal{L}(E)\right)$ is the Schwartz space of rapidly decreasing vector-valued smooth functions on $\mathbb{R}$. Then, we can write
\begin{equation*}
U(g)(x)=\left( e^{\left\vert .\right\vert B_{+}}\ast g\right) (x) = \left(g \ast e^{\left\vert .\right\vert B_{+}}\right) (x), \quad x \in \mathbb{R}.
\end{equation*}
This abstract convolution is well defined, see \cite{amann}.

Recall the abstract Fourier transform $F$ defined by
\begin{equation*}
F(\psi)(x) = \displaystyle \int_{-\infty }^{+\infty }e^{-2i\pi \xi x}\psi (\xi)d\xi,
\end{equation*}
for all $\psi \in L^{1}\left(\mathbb{R};\mathcal{L}(E\right))$ and the well known property
\begin{equation*}
F^{-1}\left( F(\phi \right) )=\phi ,
\end{equation*}
for all $\phi \in \mathscr{S}\left( \mathbb{R};\mathcal{L}(E\right))$. We have
\begin{eqnarray*}
F\left( e^{\left\vert .\right\vert B_{+}}\right) (\xi ) &=&\displaystyle \int_{-\infty
}^{0}e^{-2i\pi \xi x}e^{-xB_{+}}dx+\displaystyle \int_{0}^{+\infty }e^{-2i\pi \xi
x}e^{xB_{+}}dx \\
&=&-\left( B_{+}+2i\pi \xi I\right) ^{-1}-\left( B_{+}-2i\pi \xi I\right)
^{-1} \\
&=&-2B_{+}\left( B_{+}+2i\pi \xi I\right) ^{-1}\left( B_{+}-2i\pi \xi
I\right) ^{-1};
\end{eqnarray*}
the integrals are absolutely convergent from \eqref{YakubovEstimate}; the last equality holds from the resolvent identity. In virtue of Theorem 3.6, p. 17 in \cite{amann}, we obtain 
\begin{equation*}
U(g)(x) = F^{-1}\left( F\left( e^{-\left\vert .\right\vert B_{+}}\right)
F(g)\right) (x)=F^{-1}\left( m F(g)\right) (x),
\end{equation*}
with the Fourier multiplier
$$m(\xi ) = -2B_{+}\left( B_{+}+2i\pi \xi I\right) ^{-1}\left( B_{+}-2i\pi \xi
I\right) ^{-1} \in \mathcal{L}(E).$$
Using estimate (29), p. 14, in \cite{FaviniLabbasMaingotTorel}, we obtain
\begin{eqnarray*}
\left\Vert m(\xi )\right\Vert_{\mathcal{L}(E)} &=&\left\Vert -2B_{+}\left(
B_{+}+2i\pi \xi I\right) ^{-1}\left( B_{+}-2i\pi \xi I\right)
^{-1}\right\Vert \\
&\leqslant &2 C \left\Vert \left( B_{+}-2i\pi \xi I\right) ^{-1}\right\Vert \\
&\leqslant &\frac{2C}{\sqrt{\dfrac{\left\vert \lambda +r_{+}\right\vert }{d_{+}}+1}+\left\vert 2i\pi \xi \right\vert }.
\end{eqnarray*}
Then
\begin{equation*}
\left\Vert m(\xi )\right\Vert _{\mathcal{L}(E)}\leqslant \frac{2C\sqrt{d_{+}}}{\sqrt{\left\vert \lambda +r_{+}\right\vert +d_{+}}+2\pi \sqrt{d_{+}}\left\vert \xi
\right\vert } \leqslant \frac{2C\sqrt{d_{+}}}{\sqrt{\left\vert \lambda +r_{+}\right\vert +d_{+}}},
\end{equation*}
hence
\begin{equation*}
\underset{\xi \in \mathbb{R}}{\sup }\left\Vert m(\xi )\right\Vert _{\mathcal{L}(E)}\leqslant \frac{2C\sqrt{d_{+}}}{\sqrt{\left\vert \lambda +r_{+}\right\vert +d_{+}}}.
\end{equation*}
Now, we must estimate
\begin{equation*}
\underset{\xi \in \mathbb{R}}{\sup }\left\vert \xi \right\vert \left\Vert m'(\xi )\right\Vert
_{\mathcal{L}(E)}.
\end{equation*}
Due to the analyticity of the resolvent operator of $B_{+}$ on the imaginary
axis, it follows that 
\begin{equation*}
m\in C^{\infty }(\mathbb{R},\mathcal{L}(E)),
\end{equation*}
and
\begin{equation*}
m'(\xi )=2i\pi \left( B_{+}+2i\pi \xi I\right) ^{-2}-2i\pi \left(
B_{+}-2i\pi \xi I\right) ^{-2}.
\end{equation*}
Therefore, as above, we have
\begin{eqnarray*}
\left\vert \xi \right\vert \left\Vert m'(\xi )\right\Vert _{\mathcal{L}(E)}
&\leqslant &\left\Vert 2i\pi \xi \left( B_{+}+2i\pi \xi I\right)
^{-2}\right\Vert +\left\Vert 2i\pi \xi \left( B_{+}-2i\pi \xi I\right)
^{-2}\right\Vert \medskip \\
& \leqslant &\left\Vert 2i\pi \xi \left( B_{+}+2i\pi \xi I\right) ^{-1}\right\Vert
\left\Vert \left( B_{+}+2i\pi \xi I\right) ^{-1}\right\Vert \\
&&+\left\Vert 2i\pi \xi \left( B_{+}-2i\pi \xi I\right) ^{-1}\right\Vert
\left\Vert \left( B_{+}-2i\pi \xi I\right) ^{-1}\right\Vert \\
&\leqslant &\frac{2C}{\sqrt{\dfrac{\left\vert \lambda +r_{+}\right\vert }{
d_{+}}+1}+\left\vert 2i\pi \xi \right\vert },
\end{eqnarray*}
so
\begin{equation*}
\underset{\xi \in \mathbb{R}}{\sup }\left\vert \xi \right\vert \left\Vert m'(\xi )\right\Vert_{\mathcal{L}(E)}\leqslant \frac{2C\sqrt{d_{+}}}{\sqrt{\left\vert \lambda
+r_{+}\right\vert +d_{+}}}.
\end{equation*}
We do similarly with $\xi \longmapsto \left(\xi m'(\xi)\right)'$. Thus, from \cite{weis}, Proposition 2.5, p. 739, the sets
$$\left\{m(\xi), ~\xi \in \mathbb{R}\setminus\{0\}\right\} \quad \text{and} \quad \left\{\xi m'(\xi), ~\xi \in \mathbb{R}\setminus\{0\}\right\},$$
are R-bounded. Moreover, applying Theorem 3.4, p. 746 in \cite{weis}, we obtain 
\begin{eqnarray*}
\left\Vert U(g)\right\Vert _{L^{p}(0,L;E)} &=&\left\Vert F^{-1}\left(m F(g)\right) \right\Vert _{L^{p}(\mathbb{R};E)} \\
&\leqslant & C \left[ \underset{\xi \in \mathbb{R}}{\sup }\left\Vert m(\xi )\right\Vert +\underset{\xi \in \mathbb{R}}{\sup }\left\vert \xi \right\vert \left\Vert m'(\xi )\right\Vert \right] \left\Vert g\right\Vert _{L^{p}(0,L;E)} \\
&\leqslant &\frac{C \sqrt{d_{+}}}{\sqrt{\left\vert \lambda
+r_{+}\right\vert +d_{+}}}\left\Vert g\right\Vert _{L^{p}(0,L;E)},
\end{eqnarray*}
for all $g\in \mathcal{D}\left( 0,L;E\right).$ The same estimate is true
for all $g\in L^{p}(0,L;E)$ by density. 
\end{proof}
We will need also the following result, given by Lemma 4.12 in \cite{LMT}.
\begin{lemma}\label{Lem estim exp int X 4}
Let $g_+\in L^{p}\left(0,L;E\right)$ and $g_-\in L^{p}\left( -\ell ,0;E\right)$ with $1<p<+\infty$. Then, there exists $C > 0$ independent of $\lambda$
such that
\begin{enumerate}
\item $\dis \left\|e^{(\textbf{•}+\ell)B_-} \int_{-\ell}^0 e^{(t+\ell)B_-} g_-(t)~ dt\right\|_{L^p(-\ell,0;E)} \leqslant \frac{C \sqrt{d_-}}{\sqrt{d_- + |\lambda + r-|}} \left\|g_-\right\|_{L^p(-\ell,0;E)},$ 

\item $\dis \left\|e^{(\textbf{•}+\ell)B_-} \int_{-\ell}^0 e^{-t B_-} g_-(t)~ dt \right\|_{L^p(-\ell,0;E)} \leqslant \frac{C \sqrt{d_-}}{\sqrt{d_- + |\lambda + r_-|}} \left\|g_-\right\|_{L^p(-\ell,0;E)},$

\item $\dis \left\|e^{-\textbf{•} B_-} \int_{-\ell}^0 e^{-t B_-} g_-(t)~ ds\right\|_{L^p(-\ell,0;E)} \leqslant \frac{C \sqrt{d_-}}{\sqrt{d_- + |\lambda + r_-|}} \left\|g_-\right\|_{L^p(-\ell,0;E)},$

\item $\dis \left\|e^{-\textbf{•} B_-} \int_{-\ell}^0 e^{(t+\ell)B_-} g_-(t)~ dt\right\|_{L^p(-\ell,0;E)} \leqslant \frac{C \sqrt{d_-}}{\sqrt{d_- + |\lambda + r_-|}} \left\|g_-\right\|_{L^p(-\ell,0;E)},$

\item $\dis \left\|e^{\textbf{•} B_+} \int_0^L e^{t B_+} g_+(t)~ dt\right\|_{L^p(0,L;E)} \leqslant \frac{C \sqrt{d_+}}{\sqrt{d_+ + |\lambda + r_+|}} \left\|g_+\right\|_{L^p(0,L;E)},$ 

\item $\dis \left\|e^{\textbf{•} B_+} \int_0^L e^{(L-t)B_+} g_+(t)~ dt \right\|_{L^p(0,L;E)} \leqslant \frac{C \sqrt{d_+}}{\sqrt{d_+ + |\lambda + r_+|}} \left\|g_+\right\|_{L^p(0,L;E)},$

\item $\dis \left\|e^{(L-\textbf{•})B_+} \int_0^L e^{(L-t)B_+} g_+(t)~ ds\right\|_{L^p(0,L;E)} \leqslant \frac{C \sqrt{d_+}}{\sqrt{d_+ + |\lambda + r_+|}} \left\|g_+\right\|_{L^p(0,L;E)},$

\item $\dis \left\|e^{(L-\textbf{•})B_+} \int_0^L e^{tB_+} g_+(t)~ dt\right\|_{L^p(0,L;E)} \leqslant \frac{C \sqrt{d_+}}{\sqrt{d_+ + |\lambda + r_+|}} \left\|g_+\right\|_{L^p(0,L;E)}.$

\end{enumerate}
\end{lemma}

\subsubsection{Estimate of $\|w\|_{L^p(-\ell,L;E)}$}

We have to estimate
\begin{equation*}
\left\Vert w_{-}\right\Vert_{L^{p}(-\ell ,0;E)} + \left\Vert w_{+}\right\Vert _{L^{p}(0,L;E)},
\end{equation*}
where, due to Proposition \ref{Prop Existence-Unicité w}, $w$ is the unique solution of \eqref{SpectralEquation}. Thus, it suffices to estimate $\left\Vert w_{-}\right\Vert_{L^{p}(-\ell ,0;E)}$. The same techniques apply to $\left\Vert w_{+}\right\Vert_{L^{p}(0,L;E)}$.

We have
\begin{eqnarray*}
w_{-}\left( x\right) &=&D^{-1}\left( e^{-xB_{-}}-e^{(x+2\ell )B_{-}}\right) B_{-}^{-1}\left(I+e^{2LB_{+}}\right) (\Pi ^{\prime }) \\
&&-q_{+}D^{-1}\left( e^{-xB_{-}}-e^{(x+2\ell )B_{-}}\right)
B_{-}^{-1}B_{+}^{-1}\left( I-e^{2LB_{+}}\right) (\Pi ^{\prime })\medskip \\
&&+q_{-}D^{-1}\left( e^{(x+2\ell )B_{-}}-e^{-xB_{-}}\right)
B_{+}^{-1}B_{-}^{-1}(I-e^{2LB_{+}})(\Pi ^{\prime \prime })\medskip \\
&&-e^{(x+\ell )B_{-}}v_{-}(g_{-})(-\ell ) \medskip \\
&& + v_{-}(g_{-})(x) \medskip \\
&=& \sum_{i=1}^{5} a_{i}(x),
\end{eqnarray*}
where
$$\left\{\begin{array}{lll}
\left( \Pi'\right) &=&v_{-}'(g_{-})(0)-B_{-}e^{\ell
B_{-}}v_{-}(g_{-})(-\ell
)+q_{-}e^{LB_{+}}v_{+}(g_{+})(L)-q_{-}v_{+}(g_{+})(0) \medskip \\
&&-q_{-}e^{\ell B_{-}}v_{-}(g_{-})(-\ell )+q_{-}v_{-}(g_{-})(0), \\ \\

\left( \Pi''\right) &=&-v_{+}'(g_{+})(0)-B_{+}e^{LB_{+}}v_{+}(g_{+})(L)-q_{+}e^{LB_{+}}v_{+}(g_{+})(L)+q_{+}v_{+}(g_{+})(0)
\medskip \\
&&+q_{+}e^{\ell B_{-}}v_{-}(g_{-})(-\ell )-q_{+}v_{-}(g_{-})(0),
\end{array}\right.$$
and
\begin{equation*}
\left\{ \begin{array}{lll}
v_{-}(g_{-})(x) & = & \displaystyle \dfrac{1}{2}B_{-}^{-1}\displaystyle \int_{-\ell
}^{x}e^{(x-t)B_{-}}g_{-}(t)~ dt+\dfrac{1}{2}B_{-}^{-1} \int_{x}^{0}e^{(t-x)B_{-}}g_{-}(t)~ dt, \medskip \\ 
v_{+}(g_{+})(x) & = & \displaystyle\dfrac{1}{2}B_{+}^{-1} \int_{0}^{x}e^{(x-t)B_{+}}g_{+}(t)~ dt+ \dfrac{1}{2}B_{+}^{-1} \int_{x}^{L}e^{(t-x)B_{+}}g_{+}(t)~dt, \medskip \\ 
v'_{-}(g_{-})(x) & = & \displaystyle\dfrac{1}{2} \int_{-\ell}^{x}e^{(x-t)B_{-}}g_{-}(t)~dt-\dfrac{1}{2} \int_{x}^{0}e^{(t-x)B_{-}}g_{-}(t)~dt, \medskip \\ 
v'_{+}(g_{+})(x) & = & \displaystyle\dfrac{1}{2} \int_{0}^{x}e^{(x-t)B_{+}}g_{+}(t)~dt - \dfrac{1}{2} \int_{x}^{L}e^{(t-x)B_{+}}g_{+}(t)~dt.
\end{array}\right.
\end{equation*}
We will focus ourselves, for instance, on the first term $a_1$, that is
$$a_1(x) = D^{-1}\left( e^{-xB_{-}}-e^{(x+2\ell )B_{-}}\right) B_{-}^{-1}\left( I+e^{2LB_{+}}\right) (\Pi').$$
After replacing $(\Pi')$ by it expression, we obtain explicitly 
\begin{eqnarray*}
a_{1}(x) &=& D^{-1}B_{-}^{-1}e^{-xB_{-}}v'_{-}(g_{-})(0)+D^{-1}B_{-}^{-1}e^{-xB_{-}}e^{2LB_{+}}v'_{-}(g_{-})(0) \\
&&-D^{-1}e^{(x+2\ell )B_{-}}B_{-}^{-1}v_{-}^{\prime
}(g_{-})(0)-D^{-1}e^{(x+2\ell )B_{-}}B_{-}^{-1}e^{2LB_{+}}v_{-}^{\prime
}(g_{-})(0) \\
&&-D^{-1}e^{-xB_{-}}e^{\ell B_{-}}v_{-}(g_{-})(-\ell
)-D^{-1}e^{-xB_{-}}e^{2LB_{+}}e^{\ell B_{-}}v_{-}(g_{-})(-\ell ) \\
&&+D^{-1}e^{(x+2\ell )B_{-}}e^{\ell B_{-}}v_{-}(g_{-})(-\ell
)+D^{-1}e^{(x+2\ell )B_{-}}e^{2LB_{+}}e^{\ell B_{-}}v_{-}(g_{-})(-\ell ) \\
&&+q_{-}D^{-1}B_{-}^{-1}e^{-xB_{-}}e^{LB_{+}}v_{+}(g_{+})(L)+q_{-}D^{-1}B_{-}^{-1}e^{-xB_{-}}e^{3LB_{+}}v_{+}(g_{+})(L)
\\
&&-q_{-}D^{-1}e^{(x+2\ell
)B_{-}}B_{-}^{-1}e^{LB_{+}}v_{+}(g_{+})(L)-q_{-}D^{-1}e^{(x+2\ell
)B_{-}}B_{-}^{-1}e^{3LB_{+}}v_{+}(g_{+})(L) \\
&&-q_{-}D^{-1}e^{-xB_{-}}B_{-}^{-1}v_{+}(g_{+})(0)-q_{-}D^{-1}e^{-xB_{-}}B_{-}^{-1}e^{2LB_{+}}v_{+}(g_{+})(0)
\\
&&+q_{-}D^{-1}e^{(x+2\ell
)B_{-}}B_{-}^{-1}v_{+}(g_{+})(0)+q_{-}D^{-1}e^{(x+2\ell
)B_{-}}B_{-}^{-1}e^{2LB_{+}}v_{+}(g_{+})(0) \\
&&-q_{-}D^{-1}e^{-xB_{-}}B_{-}^{-1}e^{\ell B_{-}}v_{-}(g_{-})(-\ell
)-q_{-}D^{-1}e^{-xB_{-}}B_{-}^{-1}e^{2LB_{+}}e^{\ell
B_{-}}v_{-}(g_{-})(-\ell ) \\
&&+q_{-}D^{-1}e^{(x+2\ell )B_{-}}B_{-}^{-1}e^{\ell B_{-}}v_{-}(g_{-})(-\ell
)+q_{-}D^{-1}e^{(x+2\ell )B_{-}}B_{-}^{-1}e^{2LB_{+}}e^{\ell
B_{-}}v_{-}(g_{-})(-\ell ) \\
&&+q_{-}D^{-1}e^{-xB_{-}}B_{-}^{-1}v_{-}(g_{-})(0)+q_{-}D^{-1}e^{-xB_{-}}B_{-}^{-1}e^{2LB_{+}}v_{-}(g_{-})(0)
\\
&&-q_{-}D^{-1}e^{(x+2\ell
)B_{-}}B_{-}^{-1}v_{-}(g_{-})(0)-q_{-}D^{-1}e^{(x+2\ell
)B_{-}}B_{-}^{-1}e^{2LB_{+}}v_{-}(g_{-})(0) \medskip \\ 
& = & \displaystyle\sum_{k=1}^{24} b_k (x).
\end{eqnarray*}
Let us estimate, for instance, some terms. The others can be treated analogously. 
\begin{eqnarray*}
\left\|b_1(.)\right\|_{L^p(-\ell, 0;E)} & = & \left\Vert D^{-1}B_{-}^{-1}e^{- \textbf{•} B_{-}}v'_{-}(g_{-})(0)\right\Vert _{L^{p}\left( -\ell ,0;E\right) } \\
&=&\frac{1}{2} \left\Vert D^{-1}B_{-}^{-1}e^{- \textbf{•} B_{-}}\displaystyle \int_{-\ell }^{0}e^{-tB_{-}}g_{-}(t)dt\right\Vert_{L^{p}\left( -\ell ,0;E\right) } \\
&\leqslant &\frac{1}{2}\left\Vert D^{-1} \right\Vert_{\mathcal{L}(E)}\left\Vert B_{-}^{-1}\right\Vert_{\mathcal{L}(E)}\left[ \displaystyle \int_{-\ell }^{0}\left\Vert e^{-xB_{-}}\displaystyle \int_{-\ell}^{0}e^{-tB_{-}}g_{-}(t)dt\right\Vert_{E}^{p}dx\right]^{\frac{1}{p}};
\end{eqnarray*}
then, from \eqref{estim B+-}, \refP{Prop inv det} and Lemma \ref{Lem estim exp int X 4} statement 3. (with $[a,b]=[-\ell,0]$), there exists a constant $C >0$ independent of $\lambda$ such that
$$\begin{array}{lll}
\left\|b_1(.)\right\|_{L^p(-\ell, 0;E)} & \leqslant & \displaystyle\frac{C \sqrt{d_-} \sqrt{d_-}}{\sqrt{d_- +\left\vert \lambda + r_-\right\vert} \sqrt{d_- +\left\vert \lambda + r_-\right\vert}}\left\Vert g_{-}\right\Vert_{L^{p}\left(-\ell,0;E\right)} \medskip \\
& = & \displaystyle \frac{C d_-}{d_- +\left\vert \lambda + r_-\right\vert}\left\Vert g_{-}\right\Vert_{L^{p}\left(-\ell,0;E\right)}.
\end{array}$$
For the term $b_6$, we have 
\begin{eqnarray*}
\left\| b_6(.)\right\|_{L^p(-\ell, 0;E)} & = & \left\Vert D^{-1}e^{-\textbf{•} B_{-}}e^{2LB_{+}}e^{\ell B_{-}}v_{-}(g_{-})(-\ell )\right\Vert _{L^{p}\left( -\ell ,0;E\right) } \\
& = & \displaystyle  \left\Vert \frac{1}{2}D^{-1}B_{-}^{-1}e^{2LB_{+}}e^{2\ell
B_{-}}\displaystyle e^{-\textbf{•} B_-}\int_{-\ell }^{0}e^{t B_{-}}g_{-}(t)dt\right\Vert _{L^{p}\left( -\ell ,0;E\right) };
\end{eqnarray*}
using \eqref{estim B+-} and the fact that
$$\left\|e^{2L B_+}\right\|_{\mathcal{L}(E)} \leqslant C e^{- 2cL \left\vert \lambda_+ + \rho_+ \right\vert^{1/2}} \quad \text{and} \quad \left\|e^{2 \ell B_-}\right\|_{\mathcal{L}(E)} \leqslant C e^{- 2c \ell \left\vert \lambda_- + \rho_- \right\vert^{1/2}},$$
see \eqref{YakubovEstimate}, we obtain the existence of a constant $C>0$ independent of $\lambda$ such that 
$$\left\|b_6(.)\right\|_{L^p(-\ell, 0;E)} \leqslant \frac{C d_-}{d_- +\left\vert \lambda + r_-\right\vert}\left\Vert g_{-}\right\Vert_{L^{p}\left(-\ell,0;E\right)}.$$
For the term $b_{16}$, we have
\begin{eqnarray*}
\left\|b_{16}(.)\right\|_{L^p(-\ell,0;E)} & = & \left\Vert q_{-}D^{-1}B_{-}^{-1}e^{2\ell B_{-}}e^{2LB_{+}}e^{\textbf{•} B_{-}}v_{+}(g_{+})(0)\right\Vert _{L^{p}\left( -\ell ,0;E\right) } \\
&=&\left\Vert \frac{1}{2}q_{-}D^{-1}B_{-}^{-1}B_{+}^{-1}e^{2\ell
B_{-}}e^{2LB_{+}}e^{\textbf{•} B_{-}}\displaystyle \int_{0}^{L}e^{tB_{+}}g_{+}(t)dt\right\Vert _{L^{p}\left(-\ell ,0;E\right) }; 
\end{eqnarray*}
thus, using the same arguments as above, there exists a constant $C>0$ independent of $\lambda$ such that 
$$\left\|b_{16}(.)\right\|_{L^p(-\ell,0;E)} \leqslant \frac{q \,C \sqrt{d_+/d_-}}{\sqrt{d_- + |\lambda + r_-|} \sqrt{d_+ + |\lambda + r_+|}} \left\Vert e^{\textbf{•} B_{-}}\displaystyle \int_{0}^{L}e^{tB_{+}}g_{+}(t)dt\right\Vert _{L^{p}\left(-\ell ,0;E\right) }.$$
Now, the boundedness of the semigroup $e^{\textbf{•}B_-}$ and the Hölder inequality lead us to obtain
$$\left\|b_{16}(.)\right\|_{L^p(-\ell,0;E)} \leqslant \frac{q\, C \sqrt{d_+/d_-}}{\sqrt{d_- + |\lambda + r_-|} \sqrt{d_+ + |\lambda + r_+|}} \left\|g_{+}\right\|_{L^{p}(0,L;E) }.$$
In the same way, for the term $b_{12}$, we have
\begin{eqnarray*}
\left\|b_{12}(.)\right\|_{L^p(-\ell,0;E)} & = & \left\Vert q_{-}D^{-1}B_{-}^{-1}e^{(\textbf{•} + 2\ell)B_{-}}e^{3LB_{+}}v_{+}(g_{+})(L)\right\Vert _{L^{p}\left( -\ell ,0;E\right)} \\
&=&\left\Vert \frac{1}{2}D^{-1}B_{-}^{-1}B_{+}^{-1}e^{2\ell B_{-}}e^{4LB_{+}}e^{\textbf{•} B_{-}}\displaystyle \int_{0}^{L}e^{-tB_{+}}g_{+}(t)dt\right\Vert _{L^{p}\left(
-\ell ,0;E\right) };
\end{eqnarray*}
then
$$\left\|b_{12}(.)\right\|_{L^p(-\ell,0;E)} \leqslant \frac{q \,C \sqrt{d_+/d_-}}{\sqrt{d_- + |\lambda + r_-|} \sqrt{d_+ + |\lambda + r_+|}} \left\|g_{+}\right\|_{L^{p}(0,L;E) }.$$
Therefore, we can conclude that there exists $C>0$, independent of $\lambda$ such that
\begin{equation*}
\begin{array}{lll}
\left\| a_1 (.)\right\|_{L^{p}(-\ell,L;E)} & \leqslant & \displaystyle \frac{C d_-}{d_- +\left\vert \lambda + r_-\right\vert}\left\Vert g_{-}\right\Vert_{L^{p}\left(-\ell,0;E\right)} \medskip \\
&& + \displaystyle \frac{q \,C\, \sqrt{d_+/d_-}}{\sqrt{d_- + |\lambda + r_-|} \sqrt{d_+ + |\lambda + r_+|}} \left\|g_{+}\right\|_{L^{p}(0,L;E)}.
\end{array}
\end{equation*}
The same techniques as above lead us to obtain similar estimates for the terms $a_{i}$, $i=2,...,4$ in $w_{-}$. 

For the convolution term $v_{-}(g_{-})(.)$, using \eqref{estim B+-} and Proposition \ref{Prop Estimate U et V}, we have 
\begin{equation*}
\left\Vert v_{-}(g_{-})(.)\right\Vert _{L^{p}\left( -\ell ,0;E\right)}\leqslant \frac{C d_-}{d_- + \left\vert \lambda +r_{-}\right\vert }\left\Vert g_{-}\right\Vert _{L^{p}\left( -\ell,0;E\right) }.
\end{equation*}
Hence we can conclude that there exists of a constant $C>0$, independent of $\lambda$, such that
\begin{equation*}
\left\Vert w_{-}\right\Vert _{L^{p}\left( -\ell ,0,;E\right)}\leqslant \frac{C}{\left\vert \lambda \right\vert } \left(\left\|g_-\right\|_{L^p(-\ell,0;E)} + \left\|g_+\right\|_{L^p(0,L;E)} \right).
\end{equation*}
Using the same calculus as above, we obtain the existence of a constant $C>0$, independent of $\lambda$, such that
\begin{equation*}
\left\Vert w_{+}\right\Vert _{L^{p}\left( 0,L;E\right)}\leqslant \frac{C}{\left\vert \lambda \right\vert } \left(\left\|g_-\right\|_{L^p(-\ell,0;E)} + \left\|g_+\right\|_{L^p(0,L;E)} \right).
\end{equation*}
Summing up, we obtain
\begin{equation*}
\left\Vert w\right\Vert _{L^p(-\ell,L;E)} \leqslant \frac{C}{\left\vert \lambda \right\vert } \left\Vert g\right\Vert_{L^p(-\ell,L;E)}.
\end{equation*}
Then, we conclude that there exists $C>0$, such that, for all $\lambda \in \overline{S_{\frac{\pi}{2}}}$, we have
\begin{equation*}
\left\Vert \left( \mathcal{S}-\lambda I\right) ^{-1}\right\Vert _{\mathcal{L}\left(
L^{p}(-\ell ,L;E)\right) } \leqslant \frac{C}{|\lambda|},
\end{equation*}
which implies that $\mathcal{S}$ generates a strongly continuous analytic semigroup $\left(e^{t\mathcal{S}}\right)_{t\geqslant 0}$ in $L^{p}(-\ell ,L;E)$, see for instance \cite{tanabe}, Theorem 3.3.1, p. 68 and Remark 3.3.2, p. 69.

\begin{remark}
We have the same conclusion if we replace $\lambda \in \overline{S_{\frac{\pi}{2}}}$ by 
$$\lambda \in S_{\pi-\varepsilon_0} \setminus B(0,R),$$
due to \refP{Prop f>0}.
\end{remark}

\section*{Declarations}

\textbf{Ethical Approval:} Not applicable.
 
\noindent\textbf{Competing interests:} There is no competing interests.

\noindent\textbf{Authors' contributions:} All the authors have written and reviewed this manuscript.

\noindent\textbf{Funding:} Not applicable.

\noindent\textbf{Availability of data and materials:} Not applicable.

\section*{Acknowledgments} 

\noindent The authors would like to thank very much the referee for the valuable comments and corrections which have helped us a lot to improve this paper.

\end{document}